\documentclass{amsart}

\usepackage[all]{xy}

\usepackage[a4paper, top=2cm, bottom=2cm, left=3cm, right=3cm, heightrounded, marginparwidth=3.5cm, marginparsep=.2cm, centering]{geometry}
\usepackage{hyperref}
\hypersetup{colorlinks, citecolor=blue, filecolor=black, linkcolor=black, urlcolor=black}
\usepackage{cleveref}
\usepackage{calligra,mathrsfs}
\usepackage{changepage}

\crefformat{footnote}{#2\footnotemark[#1]#3}

\usepackage{array}
\usepackage{multirow}
\usepackage{makecell}

\usepackage[shortlabels]{enumitem}
\usepackage{verbatim}
\usepackage{tikz-cd}
\usepackage{bbold}
\usepackage{amssymb}
\usepackage{amsaddr}
\usepackage{todonotes}
\usepackage[T1]{fontenc}
\usepackage{footnote}
\makesavenoteenv{tabular}

\usepackage[bbgreekl]{mathbbol}
\usepackage{amsfonts,amsthm}
\usepackage{stmaryrd}
\usepackage{scalerel}

\DeclareSymbolFontAlphabet{\mathbb}{AMSb}
\DeclareSymbolFontAlphabet{\mathbbl}{bbold}
\newcommand{\prism}{{\scaleobj{1.25}{\mathbbl{\Delta}}}}

\DeclareRobustCommand{\SkipTocEntry}[5]{}

\theoremstyle{plain}
\newtheorem{theorem}{Theorem}[section]
\newtheorem{proposition}[theorem]{Proposition}
\newtheorem{lemma}[theorem]{Lemma}
\newtheorem*{claim*}{Claim}
\newtheorem{corollary}[theorem]{Corollary}

\theoremstyle{definition}

\newtheorem{Setup}[theorem]{Setup}
\newtheorem{definition}[theorem]{Definition}

\newtheorem{remark}[theorem]{Remark}

\newcommand{\ad}{\mathrm{ad}}
\newcommand{\an}{\mathrm{an}}

\newcommand{\Aut}{\mathrm{Aut}}

\newcommand{\cts}{\mathrm{cts}}
\newcommand{\cycl}{\mathrm{cycl}}

\newcommand{\dR}{\mathrm{dR}}

\newcommand{\End}{\mathrm{End}}

\newcommand{\et}{\mathrm{{\acute{e}t}}}

\newcommand{\Fsm}{\mathrm{F\text{-}sm}}

\newcommand{\Gal}{\mathrm{Gal}}

\newcommand{\GL}{\mathrm{GL}}
\newcommand{\Symvw}{{\mathcal S}_p}

\newcommand{\Hom}{\mathrm{Hom}}
\newcommand{\HT}{\mathrm{HT}}
\newcommand{\HTlog}{\mathrm{HTlog}}
\newcommand{\Hsm}{\mathrm{H\text{-}sm}}
\newcommand{\zHsm}{{z\text{-}\mathrm{H\text{-}sm}}}
\newcommand{\Id}{\mathrm{Id}}

\renewcommand{\inf}{\mathrm{inf}}
\newcommand{\id}{\operatorname{id}}

\newcommand{\Lft}{\mathcal{L}\mathrm{ft}}
\newcommand{\LS}{\mathrm{LS}}

\newcommand{\N}{\mathbb{N}}

\newcommand{\Perf}{\mathrm{Perf}}

\newcommand{\Q}{\mathbb{Q}}

\newcommand{\sm}{\mathrm{sm}}
\newcommand{\Spa}{\mathrm{Spa}}

\newcommand{\Spf}{\mathrm{Spf}}

\newcommand{\WCart}{\mathrm{WCart}}
\newcommand{\Z}{\mathbb{Z}}

\newcommand{\Pic}{\mathrm{Pic}}
\renewcommand{\O}{\mathcal{O}}
\renewcommand{\Vec}{\mathrm{Vec}}
\newcommand{\wt}{\widetilde}
\newcommand{\wtOm}{\wt{\Omega}}
\newcommand{\tf}{[\tfrac{1}{p}]}
\newcommand{\X}{{\mathcal X}}

\newcommand{\Bun}{\mathrm{Bun}}
\newcommand{\Higgs}{\mathrm{Higgs}}

\newcommand{\sBun}{\text{\normalfont v}\mathscr{B}\hspace{-0.1em}\mathrm{\textit{un}}}
\newcommand{\sHiggs}{\mathscr H\hspace{-0.1em}\mathrm{\textit{iggs}}}

\newcommand*\isomarrow{%
	\xrightarrow{\raisebox{-0.35em}{\smash{\ensuremath{\sim}}}}
}  

\newcommand{\xisomarrow}[1]{%
	\xrightarrow[\raisebox{0.3em}{\smash{\ensuremath{\sim}}}]{#1}
}

\begin{document}

	\title[The small \MakeLowercase{p}-adic Simpson correspondence in terms of moduli spaces]{The small $p$-adic Simpson correspondence in terms of moduli spaces}
	
	\author[J. Ansch\"utz, B. Heuer, A.-C. Le Bras]{Johannes Ansch\"utz, Ben Heuer, Arthur-C\'esar Le Bras}

	\begin{abstract}
          For any rigid space  over a perfectoid extension of $\Q_p$ that admits a liftable smooth formal model, we construct an isomorphism between the moduli stacks of Hitchin-small Higgs bundles and Hitchin-small v-vector bundles. This constitutes a moduli-theoretic improvement of  the small $p$-adic Simpson correspondence of Faltings, Abbes--Gros, Tsuji and Wang.  Our construction is based on the Hodge--Tate stack of Bhatt--Lurie. We also prove an analogous correspondence in the arithmetic setting of rigid spaces of good reduction over $p$-adic fields.
	\end{abstract}
	
	\maketitle

\section{Introduction}
\label{sec:introduction-1}

\subsection{The $p$-adic Simpson correspondence}
	Let $K$ be a complete algebraically closed extension of $\Q_p$. Let $\X$ be a smooth rigid space over $K$. Let $\X_v$ be the v-site of $\X$ as defined in \cite{scholze_etale_cohomology_of_diamonds}. Reinterpreted in terms of Scholze's perfectoid foundations for $p$-adic Hodge theory, the $p$-adic Simpson correspondence, pioneered independently by Deninger--Werner \cite{DeningerWerner_vb_p-adic_curves} and Faltings \cite{faltings2005p}, aims to relate the category of vector bundles on $\X_v$ to the category of Higgs bundles on $\X$. Here v-vector bundles replace Faltings' ``generalised representations'' in the rigid analytic setting.
	
	More precisely, the $p$-adic Simpson correspondence comes in two different flavours: When $\X$ is assumed to be proper,  there exists an equivalence of categories, depending on certain choices,
	\begin{equation}\label{eq:proper-p-adic-Simpson}
	\{ \text{Higgs bundles on $\X$}\}\isomarrow 	\{ \text{v-vector bundles on $\X$}\}.
	\end{equation}
	Faltings constructed this for curves in \cite{faltings2005p}, and the general case has recently been shown in \cite{heuer-proper-correspondence}.

	Second, there is also a $p$-adic Simpson correspondence for not necessarily proper $\X$, but for this one needs to restrict to ``small'' objects on either side. Here a v-vector bundle, respectively a Higgs bundle, is called ``small'' if it locally admits an integral model whose reduction is trivial modulo $p^{\alpha}$ for a certain $\alpha\in \mathbb R$ (see \Cref{def:faltings-small-v-bundles} for a precise definition).  This is inherently a statement about integral structures, and we therefore need to make additional assumptions on integral models: If $X$ is any rigid space of good reduction, Wang \cite{wang2021p} has shown that there is an equivalence of categories
	\begin{equation}\label{eq:small-p-adic-Simpson}
	 \{\text{small Higgs bundles on $\X$}\}\isomarrow\{ \text{\text{small} v-vector bundles on $\X$}\}.
	 \end{equation}
	Wang's functor is based on an earlier construction of Liu--Zhu \cite{LiuZhu_RiemannHilbert}. The first instance of such a ``small'' $p$-adic Simpson correspondence had previously been constructed by Faltings \cite{faltings2005p} for semi-stable $p$-adic schemes, then further developed by Abbes--Gros and Tsuji \cite{abbes2016p}. 
	We note that Faltings used a version of \eqref{eq:small-p-adic-Simpson} to deduce \eqref{eq:proper-p-adic-Simpson}. But we now think of \eqref{eq:proper-p-adic-Simpson} and \eqref{eq:small-p-adic-Simpson} as two logically independent statements which are both of interest in their own right, as they apply in different situations.
	
	We refer to the appendix for an overview of various instances of $p$-adic Simpson correspondences.
	\subsection{The small $p$-adic Simpson correspondence in terms of moduli spaces}\label{s:intro-small-paS-moduli}
	The first goal of this article is to use our approach via Hodge--Tate stacks from  \cite{AHLB} to construct a new relative version of the small $p$-adic Simpson correspondence that improves \eqref{eq:small-p-adic-Simpson} in various ways (see 
	\S\ref{sec:relat-prev-work}.1).
	We use this to realise the small $p$-adic Simpson correspondence as an isomorphism of moduli spaces:
	\begin{theorem}[\Cref{cor:smoothoid-case-2-comparison-of-moduli-spaces-over-the-hitchin-small-locus}]
		\label{t:intro-moduli-spaces}
		Let $K$ be any perfectoid field over $\Q_p$ and let $X$ be a smooth formal scheme over $\O_K$.  Let $\mathcal X$ be the adic generic fibre of $X$. Let $n\in \N$. Then
		any lift $\tilde{X}$ of $X$ to $A_2:=W(\O_{K^\flat})/(\ker \theta)^2$ induces an equivalence of small v-stacks
		\[
		\mathcal{S}_{\tilde{X}}\colon \sHiggs_{n}^{\Hsm} \isomarrow \sBun^{\Hsm}_n 
		\]
		between the  stacks of Hitchin-small Higgs bundles, respectively v-vector bundles, of rank $n$ on $\mathcal X$.
	\end{theorem}
	In fact, we allow lifts of $X$ to larger rings $A_2(x)\subseteq B_\dR^+/\xi^2$ at the expense of a stronger smallness condition. This has the advantage that $A_2(x)$-lifts always exist for example when $X$ is proper.
	
	We now explain the statement of \Cref{t:intro-moduli-spaces} in more detail:  Recently, \cite{heuer-sheafified-paCS} has introduced analytic moduli spaces for both sides of the $p$-adic Simpson correspondence. Namely, working in the setting of Scholze's v-stacks \cite{scholze_etale_cohomology_of_diamonds}, for $n\in \N$, we define functors
	\begin{alignat*}{5}
		\sHiggs_{n}:&\;\Perf_K\to \mathrm{Groupoids},\quad &&S\mapsto \{\text{Higgs bundles on $\mathcal X\times S$ of rank $n$}\},\\
		\sBun_{n}:&\;\Perf_K\to \mathrm{Groupoids},\quad &&S\mapsto \{\text{v-vector bundles on $\mathcal X\times S$ of rank $n$}\}.
	\end{alignat*}
	These are small v-stacks on the categories of analytic affinoid perfectoid spaces $S$ over $K$ \cite[Thm~7.18]{heuer-sheafified-paCS}. They are equipped with natural Hitchin fibrations: Indeed, there are natural morphisms
	\[
	\begin{tikzcd}[row sep =-0.1cm,column sep = 1cm]
		\sHiggs_{n}\arrow[r,"\mathcal H"]& \mathcal A_{n} & {\sBun_{n}}\arrow[l,"\wt{\mathcal H}"']
	\end{tikzcd}
	\] of small v-stacks.
Here $\mathcal A_{n}$ is the Hitchin base of rank $n$, which is the v-sheaf defined by
	\[  \textstyle\mathcal A_{n}:\Perf_K\to \mathrm{Sets},\quad S\mapsto \bigoplus\limits_{k=1}^nH^0(\X\times S,\mathrm{Sym}^k(\Omega^1_{\X}(-1))).\]
	We can now explain our notion of ``smallness'': We first define the Hitchin-small locus $	\mathcal{A}_n^{\Hsm}\subseteq  \textstyle\mathcal A_{n}$
	\[\textstyle
	\mathcal{A}_n^{\Hsm}:\Perf_K\to \mathrm{Sets},\quad S=\Spa(R,R^+)\mapsto\bigoplus\limits_{k=1}^n p^{<\frac{k}{p-1}}\cdot H^0(X\times \Spf(R^+), \mathrm{Sym}^k(\Omega^1_{X}\{-1\})).
	\]
	Here $p^{<\frac{k}{p-1}}:=\{x\in\O_K \text{ s.t. } |x|< |p|^{\frac{k}{p-1}}\}$ and $\{-1\}$ is a Breuil--Kisin--Fargues twist (see below).
	
	If $X$ is proper, then $\mathcal A_n$ is represented by an affine rigid group, and $\mathcal A_n^\Hsm$ is then represented by an open subgroup of $\mathcal A_n$. For general $X$, they still make sense as v-sheaves. 
	
	We now define a Higgs bundle on $\X\times S$ to be Hitchin-small if its image under the Hitchin fibration $\mathcal H$ lands in $\mathcal A_n^\Hsm(S)$, and similarly for v-vector bundles. We thus set
	\[\sHiggs^{\Hsm}_n:= \sHiggs_n\times_{\mathcal A_n}\mathcal A_n^\Hsm,\quad  \sBun^{\Hsm}_n:= \sBun_n\times_{\mathcal A_n}\mathcal A_n^\Hsm.\]
	This explains the objects appearing in \Cref{t:intro-moduli-spaces}. We will moreover show that the equivalence $\mathcal S_{\wt X}$ from \Cref{t:intro-moduli-spaces} commutes with the Hitchin fibrations on both sides.
	\subsection{The small $p$-adic Simpson correspondence in the relative case}
	To construct the isomorphism $\mathcal S_{\wt X}$, we can in fact work in a more general setting:
Let $X$ be any $p$-adic formal scheme that is smooth over a perfectoid $p$-adic formal scheme $\Spf(R)$.  Let $\Omega^1_{X}:=\Omega^1_{X|R}$ be the sheaf of $p$-completed K\"ahler differentials. Let $(A,I)$ be the perfect prism associated with $R$, i.e., $A/I\cong R$. For any $R$-module $M$, we denote by $M\{1\}$ the Breuil--Kisin--Fargues twist $M\otimes_RI/I^2$. As usual, one also defines $M\{-1\}:=M\otimes_R\Hom(R\{1\},R)$. Let $\mathcal X$ be the adic generic fibre of $X$. There is in this setting a good notion of Higgs bundles on $X$ defined using the sheaf $\Omega^1_{X}\{-1\}$. In \cite{AHLB}, we have used the Hodge--Tate stack $X^\HT$of Bhatt--Lurie \cite{bhatt2022prismatization} to construct a fully faithful functor
  \[ \mathrm S_{\tilde{X}}:\{\text{Hitchin-small Higgs bundles  on } \X\}\hookrightarrow \{\text{v-vector bundles  on } \X\},\]
  natural in $\tilde{X}$.
  The remaining work towards \Cref{t:intro-moduli-spaces} lies in describing the essential image:
\begin{theorem}[\Cref{t:essential-surjectivity-geom-case}]
	\label{t:relative-small-correspondence}
	The essential image of $\mathrm S_{\tilde{X}}$  is given by the category of Hitchin-small v-vector bundles, i.e.\ we have an equivalence of categories, natural in $\X$,
	\[ \mathrm{S}_{\tilde X}:{\left\{ \begin{array}{@{}c@{}l}\text{Hitchin-small}\\\text{Higgs modules on $\X$}\end{array}\right\}}\isomarrow {\left\{ \begin{array}{@{}c@{}l}\text{Hitchin-small}\\v\text{-vector bundles on $\X$}\end{array}\right\}}.\]
\end{theorem}
This implies \Cref{t:intro-moduli-spaces}. We note that
 the derived construction of \cite{AHLB} immediately shows that $\mathrm{S}_{\tilde X}$ also identifies v-cohomology on the right with Dolbeault cohomology on the left.

As an easy special case of \Cref{t:relative-small-correspondence}, we can deduce the first non-trivial example of an isomorphism between the full moduli stacks, namely when $X=\mathbb P^m_{\O_K}$ for some $m\in \N$:
\begin{corollary}
	Let $X=\mathbb P^m_{\O_K}$, then for any $n \in \N$, there is a canonical equivalence of v-stacks
	\[\mathcal{S}:
	\sHiggs_{n}\isomarrow \sBun_{n}.
	\]
\end{corollary}

This is a non-trivial statement as there are Higgs bundles with non-vanishing Higgs field on $\mathbb P^m_{\O_K}$. But we caution that we cannot expect to have such an equivalence for general $X$, see below.

\subsection{Relation to other works on the $p$-adic Simpson correspondence}
\label{sec:relat-prev-work}    
\begin{enumerate}[label=\arabic*.]
	\item \Cref{t:relative-small-correspondence} is a generalisation of Wang's small $p$-adic Simpson correspondence \eqref{eq:small-p-adic-Simpson}, which is the main result of \cite{wang2021p}. Our result is more general in four different ways:
	\begin{enumerate}
	 \item We work in a relative situation of smooth families over a perfectoid base.
	 \item Our notion of smallness is more general. For example, it also encompasses the local systems treated by Liu--Zhu \cite{LiuZhu_RiemannHilbert}, at least when $X$ has good reduction. 
	 \item We allow a more general class of lifts of $X$ to subrings of $B_{\dR}^+/\xi^2$. For example, this means that we can remove the ``liftable'' condition of \cite{wang2021p} for proper $X$.
	\item We prove a more general functoriality statement by explaining Faltings' notion of ``twisted pullback'' in terms of the Hodge--Tate stack in \S\ref{s:twisted-pullback}.
	\end{enumerate}
	While the strategies are different,  Wang's functor can be compared to ours, see \cite[\S7]{AHLB}.
  \item When  $R=\O_K$ and $K$ is the completion of an algebraic closure of a discretely valued extension of $\Q_p$ with perfect residue field, then \Cref{t:relative-small-correspondence} is closely related to the global $p$-adic Simpson correspondence for small objects of Faltings \cite{faltings2005p}, hence also to the constructions by Abbes--Gros and Tsuji \cite{abbes2016p}\cite{Tsuji-localSimpson} which it inspired (up to different technical foundations, e.g.\ these authors work in an algebraic setup, whereas we work in the rigid analytic setting). While there are some conceptual similarities in this case, especially to Tsuji's construction (we refer to \cite[\S1.6]{AHLB} for more details on these), a key new idea is to use the Hitchin fibration $\wt{\mathcal H}$ to describe the essential image in v-vector bundles.
  
  We note that the above works can more generally deal with semi-stable $X$. The reason we restrict our attention to smooth $X$ is that the geometry of the Hodge--Tate stack is currently only understood in this case.
  \item Faltings also proves a third instance of a $p$-adic Simpson correspondence, induced by a toric chart. This ``local $p$-adic Simpson correspondence'' is more basic than \eqref{eq:small-p-adic-Simpson} in that it imposes further conditions on $\X$, e.g.\ it has to be affine. But it can in fact deal with a stronger smallness condition than the global correspondence. In  \S\ref{sec:explicit-toric-charts}, we give a comparison of \Cref{t:relative-small-correspondence} to the local correspondence in the case when $\X$ is toric.
  \item To understand our ``Hitchin smallness'' condition, it is helpful to consider the case of rank $n=1$: Following \cite{heuer-v_lb_rigid}, the correspondence \eqref{eq:small-p-adic-Simpson} is then explained by the left-exact sequence
  \[ 0\to \Pic_{\et}(\X)\to \Pic_v(\X)\xrightarrow{\HTlog} H^0(\X,\Omega^1(-1))\]
  If $\X$ is proper, this is right-exact, and a splitting yields \eqref{eq:proper-p-adic-Simpson}. But for general $\X$, it is no longer right-exact, only split over an integral sublattice of $H^0(\X,\Omega^1(-1))$, see \cite[\S5]{heuer-v_lb_rigid}. This is reflected by our notion of Hitchin-smallness: $\HTlog$ is the $K$-points of the Hitchin fibration $\wt {\mathcal H}$ for $n=1$. Using this, one can see at the example of elliptic curves with supersingular reduction that the condition on Hitchin-smallness is optimal in general. We therefore think of \Cref{t:relative-small-correspondence} as being ``optimal'' except for the good reduction assumption.
  \item The isomorphism of small moduli spaces \Cref{t:intro-moduli-spaces} has previously only been constructed in two special cases: In \cite{heuer-sheafified-paCS}, there is a relative version of the ``local correspondence'' (see 3.) in terms of toric charts. Second, the case of \Cref{t:intro-moduli-spaces} of rank $n=1$ has been treated in \cite{heuer-geometric-Simpson-Pic}. This case also shows that in  the setup of \eqref{eq:proper-p-adic-Simpson}, it is no longer true that the moduli spaces $\sHiggs_n$ and $\sBun_{v,n}$ are isomorphic if one drops the smallness condition. Instead, one can hope for a comparison up to a twist by a certain moduli space of exponentials. In \cite{HX}, such a ``twisted isomorphism of moduli stacks'' between $\sHiggs_n$ and $\sBun_{v,n}$ will be constructed in higher rank when $\X$ is a smooth projective curve. 
  From this perspective, the geometric situation for the moduli spaces in the small case \eqref{eq:small-p-adic-Simpson} is in some sense better than that in the proper case \eqref{eq:proper-p-adic-Simpson}, as we can obtain an actual equivalence of moduli spaces.
  \end{enumerate}

\subsection{The arithmetic case}
\label{sec:arithmetic-case}
The second topic of this article is a new description of \mbox{v-vector} bundles in an arithmetic setup: Let now $X$ be a smooth $p$-adic formal scheme over the ring of integers $\O_K$ of a $p$-adic field $K$. Let $C$ be the completion of an algebraic closure of $K$. Choose a uniformizer $\pi \in \mathcal{O}_K$ and let $E$ be its characteristic polynomial over the maximal unramified subfield of $K$. Let $e:=E'(\pi)$, then the different ideal of $\O_K$ is $\delta_{\O_K|\Z_p}=(e)$. Let $\X$ be the adic generic fibre of $X$. In this setting, Min--Wang have  constructed a fully faithful functor \cite[$\mathrm F_S$ in Thm 1.1]{MinWang22}
\[
\mathrm{S}_\pi \colon {\{\text{Hitchin-small Higgs--Sen modules on $\X$}\}}\hookrightarrow \{\text{v-vector bundles  on } \X\}
\]
which we have reinterpreted in \cite{AHLB} in terms of the Hodge--Tate stack $X^\HT$ of $X$. Here we define:
\begin{definition}\label{d:Higgs--Sen}
 A Higgs--Sen bundle (called an ``enhanced Higgs bundle'' in \cite{MinWang22}) on $\X$ is a 
	triple $(N,\theta_{N}, \phi_N)$ consisting of 
	\begin{enumerate}
		\item a vector bundle $N$ on $\X$,
		\item a topologically nilpotent Higgs field $\theta_N: N\to N\otimes_{\O_\X} \Omega^1_{\X|K}\{-1\}$,
		\item an $\O_\X$-linear endomorphism $\phi_N\colon N\to N$  making the following diagram commute:
		\[\begin{tikzcd}
			N  & {N \otimes_{\O_\X} \Omega^1_{\X|K}\{-1\}} \\
			N & {N\otimes_{\O_\X} \Omega^1_{\X|K}\{-1\}}
			\arrow["{\theta_{N}}", from=2-1, to=2-2]
			\arrow["{\theta_{N}}", from=1-1, to=1-2]
			\arrow["{\phi_N\otimes 1-e\cdot \id}", from=1-2, to=2-2]
			\arrow["{\phi_N}"', from=1-1, to=2-1]
		\end{tikzcd}\]
	\end{enumerate}
	A Higgs--Sen bundle $(N,\theta_{N}, \phi_N)$ is called Hitchin-small if the eigenvalues of $\phi_N$ are in   $\Z+e^{-1}\cdot \mathfrak{m}_{\overline{K}}$.
	
\end{definition}

The $\{-1\}$ indicates the Breuil-Kisin twist by the invertible $\O_K$-module $\O_K\{-1\}$ defined as the pullback along $\Spf(\O_K)\to \Spf(\O_K)^\HT$ of the natural bundle $\O_{\Spf(\O_K)^{\HT}}\{-1\}$ from \cite[Example~3.5.2]{bhatt2022absolute}. We can always use our chosen uniformiser $\pi\in \O_K$ to trivialise this twist and thus ignore $\{-1\}$ in the above. But it is better to keep it, for example to explain the naturality in $K$. In particular, it has a non-trivial effect on the Galois action when passing to extensions of $K$.

The second main result of this article is a description of the essential image of $\mathrm S_\pi$. For this we use that v-vector bundles on $\Spa(K)$ are equivalent to semi-linear $p$-adic Galois representations of $\Gal(C|K)$ on finite dimensional $C$-vector spaces. Thus v-vector bundles on $\X$ can be interpreted as rigid analytic families of $p$-adic Galois representations. Taking this perspective, we define:
\begin{definition}
	A v-vector bundle $V$ on $\X$ is called Hitchin-small if for every point $\Spa(L)\to \X$  valued in a finite extension $L|K$, the Galois representation $V_x$ is ``nearly Hodge--Tate'' in the sense of Gao \cite{gao2022hodge}, i.e.\ the Hodge--Tate--Sen weights of $V_x$ are in   $\Z+\delta_{\mathcal{O}_L/\Z_p}^{-1}\cdot \mathfrak{m}_{\overline{L}}$.
\end{definition}
\begin{theorem}[\Cref{t:essential-image-vVB-arithmetic}]
	\label{t:essential-image-vVB-arithmetic-introduction}
	\begin{enumerate}
		\item 
	The essential image of the functor $\mathrm{S}_\pi$ is given by the Hitchin-small v-vector bundles, i.e.\ we have an equivalence of categories, natural in $X$,
	\[ \mathrm{S}_\pi:{\left\{ \begin{array}{@{}c@{}l}\text{Hitchin-small}\\\text{Higgs--Sen bundles on $\X$}\end{array}\right\}}\isomarrow {\left\{ \begin{array}{@{}c@{}l}\text{Hitchin-small}\\v\text{-vector bundles on $\X$}\end{array}\right\}}.\]
	\item Every v-vector bundle on $\X$ becomes Hitchin-small after a finite extension of $K$.
	\end{enumerate}
\end{theorem}
 By analogy to \S\ref{s:intro-small-paS-moduli}, the name ``Hitchin-small'' is motivated by the idea that the map
\[ \{\text{Higgs--Sen bundles on $\X$ of rank $n$}\}\to \mathbb A^n(\X), \quad (N,\theta_N,\phi_N)\mapsto \chi(\phi),\]
sending a Higgs--Sen bundle to the coefficients of the characteristic polynomial of the Sen operator, is the correct analogue of the Hitchin fibration in the arithmetic setting. Then a Higgs--Sen module on $\X$ is Hitchin-small if and only if it is sent into a certain integral sublattice of $\mathbb A^n(\X)=\O(\X)^n$. We note that the condition (3) in \Cref{d:Higgs--Sen} forces $\theta_N$ to be nilpotent, so that the usual Hitchin fibration in terms of $\theta_N$ is trivial. In other words, only $\phi_N$ has non-trivial spectral information.

\medskip 

In \cite{AHLB}, we had constructed $\mathrm{S}_\pi$ using a canonical pullback morphism $\Bun(X^\HT,\O\tf)\to \Bun(\X_v)$.  
Using \cite[Thm~6.37]{AHLB}, we can now deduce from \Cref{t:essential-image-vVB-arithmetic-introduction}.(1)  that its essential image consists precisely of the Hitchin-small objects. From  \Cref{t:essential-image-vVB-arithmetic-introduction}.(2) we therefore now obtain:
\begin{corollary}\label{c:equiv-to-Bun(Xv)}
	There is an equivalence of categories
	\[\textstyle 2\text{-}\varinjlim_{L|K} \Bun([X_{\O_L}^{\HT}/\Gal(L|K)],\O\tf)\to  \Bun(\X_v) \]
	where $L|K$ runs through the finite Galois extensions of $K$ inside $\overline{K}$.
\end{corollary}
Conceptually, this means that we can think of the inverse limit of $[X_{\O_L}^{\HT}/\Gal(L|K)]$ over $L|K$ as being a geometric object whose analytic vector bundles correspond to v-vector bundles on $\X$.

\subsection{Relation to other works in the arithmetic case}\label{s:rel-other-works-arithmetic}
\begin{enumerate}
	\item  The first $p$-adic Simpson functor in the arithmetic situation has been constructed by Tsuji \cite[\S15]{Tsuji-localSimpson}: In an algebraic setting of certain log schemes over $\O_K$, Tsuji constructs a fully faithful functor which (reinterpreted in the perfectoid language of this article) goes from v-vector bundles $V$ on $\X$ to $\Gal(C|K)$-equivariant nilpotent Higgs bundles on $\X_{C}$ \cite[Thm 15.1]{Tsuji-localSimpson}. We give a rigid-analytic variant of this construction as an intermediate step towards \Cref{t:essential-image-vVB-arithmetic-introduction}. Namely, $S_\pi^{-1}(V)$ can be described  by forming the morphism of Sen modules associated to the Galois action on the Higgs field. From this perspective, the main differences to our work in this case are that we deal more generally with the rigid-analytic setting (albeit with a good reduction assumption), and the description of the essential image of $\mathrm S_{\pi}$ in terms of Hitchin-smallness is a new contribution.
	\item In the rigid analytic setting, the functor $\mathrm S_{\pi}$ has previously been described by Min--Wang \cite{MinWang22},  but they left open the description of the essential image  \cite[Remark 7.25]{MinWang22}.  Their approach via prismatic crystals is closely related to earlier work of Tian \cite{tian2021finiteness}.
	\item When $\mathbb L$ is a $\Z_p$-local system on $\X$, then $V=\mathbb L\otimes \O$ is a v-vector bundle on $\X$. In this special case, the functor $\mathrm S_\pi$ is closely related to the well-known $p$-adic Simpson functor  of Liu--Zhu \cite[\S2]{LiuZhu_RiemannHilbert} (here the setup is slightly different: there is no good reduction assumption, but $\X$ has an arithmetic model). More precisely, the base-change to $\X_C$ of the underlying Higgs bundle of $S_{\pi}^{-1}(V)$ is the nilpotent Higgs bundle associated to $\mathbb L$ by Liu--Zhu.
	\item 	Inspired by the work of Liu--Zhu, the Sen operator  associated to $\mathbb L$ in this rigid analytic setting has recently been constructed by Shimizu \cite[\S2]{Shimizu_constancy} and Petrov \cite[\S3]{petrov2020geometrically}. 
	
	From this perspective, (up to technical differences in the setups) we generalize these functors from $\Z_p$-local systems to arbitrary v-vector bundles, leading to an equivalence.
	\item In the special case of $X=\Spf(\O_K)$, \Cref{t:essential-image-vVB-arithmetic-introduction} is a statement about Sen theory \cite{sen1980continuous}: Higgs--Sen modules are then precisely the classical Sen modules. In this case, the approach via prismatic crystals has previously been studied by Min--Wang \cite{min2021hodge} and Gao \cite[Thm~1.1.5]{gao2022hodge}, who was the first to describe the essential image of prismatic crystals in v-vector bundles on $\Spa(K)_v$ in terms of the ``nearly Hodge--Tate'' condition. In \cite{analytic_HT}, we studied the case  of  $X=\Spf(\O_K)$ via $X^\HT$, and \Cref{c:equiv-to-Bun(Xv)} generalises \cite[Thm~1.2]{analytic_HT}. 
	
	From this perspective, we can regard our results in the arithmetic case as a generalisation of these results for $\Spa(\O_K)$ to analytic families of semi-linear $\Gal(C|K)$-representations.
\end{enumerate}

To the best of our knowledge, beyond the base case $X=\Spf(\O_K)$,
\Cref{t:essential-image-vVB-arithmetic-introduction} is the first instance of a $p$-adic Simpson \textit{correspondence} in the arithmetic setting over $K$, relating v-vector bundles to purely linear algebraic data in terms of an equivalence of categories. 
\subsection*{Acknowledgements}

We would like to thank Bhargav Bhatt, Hui Gao, Tongmu He, Juan \mbox{Esteban} Rodr\'iguez Camargo, Peter Scholze, Yupeng Wang, Annette Werner, Matti W\"urthen and Bogdan Zavyalov for helpful discussions. 

\

The second author was funded by the Deutsche Forschungsgemeinschaft (DFG, German Research Foundation) -- Project-ID 444845124 -- TRR 326.

\subsection*{Notations}
\label{sec:notations}

We will use the following notations.
\begin{enumerate}
\item $p$ is a prime.
\item Let $\Z_p^\cycl:=\Z_p[\zeta_{p^\infty}]^\wedge_p$ and $\Q_p^\cycl:=\Z_p^\cycl\tf$ the cyclotomic perfectoid field. We fix a primitive $p$-th root of unity $\zeta_p \in \Z_p^\cycl$ and write $\omega:=(\zeta_p-1)^{-1} \in \Q_p^\cycl$.
  \item If $R_0 \to R$ is a morphism of $p$-complete rings, we will denote by $\Omega_{R|R_0}^1$ the \textit{$p$-completion} of the module of differential forms of $R$ over $R_0$. This construction glues to define, for any morphism $X \to X_0$ of $p$-adic formal schemes, the sheaf $\Omega_{X|X_0}^1$ of $p$-completed differential forms of $X$ over $X_0$. Its $\mathcal{O}_X$-linear dual, the tangent sheaf, is denoted by $T_{X|X_0}$. The same notation will be used for adic spaces. 
  \item If $X$ is a $p$-adic formal scheme, we denote by $X^{\prism}$ (resp. $X^{\rm HT}$) the Cartier--Witt stack of $X$ (resp. the Hodge--Tate locus in the Cartier--Witt stack of $X$, or simply Hodge--Tate stack of $X$), which is denoted $\WCart_X$ (resp. $\WCart_X^{\rm HT}$) in \cite[Definition 3.1, Construction 3.7]{bhatt2022prismatization}).
\end{enumerate}

\section{Recollections on $p$-adic Simpson functors via the Hodge--Tate stack}\label{s:recollections}
This article is based on our earlier work \cite{AHLB} about the geometry of the Hodge--Tate stack of Bhatt--Lurie \cite{bhatt2022absolute}\cite{bhatt2022prismatization}. However, we will not need to analyse the Hodge--Tate stack in this work: Instead, we only rely on the resulting constructions of \cite{AHLB}, which we shall now recall.

We begin by making precise the two different technical setups:

\begin{Setup}[geometric Setup]\label{setup:geometric}
	Let $S$ be a $p$-torsionfree $p$-adic perfectoid ring. We assume that $S$ contains a $p$-th unit root $\zeta_p$. Let  $(A,I)$ be the perfect prism associated to $S$. We set
	$A_2:=A/I^2$.
	
	Let $X$ be a $p$-adic formal scheme that is smooth over $S$. We denote by $\X$ the adic generic fibre of $X$, an analytic adic space over the affinoid perfectoid space $\Spa(S\tf,S)$.
\end{Setup}
We note that $\X$ is a sheafy analytic adic space. Indeed, it is a smoothoid adic space in the sense of  \cite[\S2]{heuer-sheafified-paCS}, and we refer to there for some basic technical properties, especially \cite[Lemma~2.6]{heuer-sheafified-paCS}.

\begin{Setup}[arithmetic Setup]\label{setup:arithmetic}
	Let $K$ be a $p$-adic field, by this we mean a discretely valued non-archimedean field extension of $\Q_p$ with perfect residue field. Let $\O_K$ be the ring of integers and fix a uniformiser $\pi$. Let $C$ be the completion of an algebraic closure of $K$.
	
	 Let $X$ be a $p$-adic formal scheme that is smooth over $\Spf(\O_K)$. Let $\X$ be its adic generic fibre.
\end{Setup}

We note that we can pass from \Cref{setup:arithmetic} to \Cref{setup:geometric} by sending $X$ to the base-change $X_{\O_C}$.

\medskip

The central object of \cite{AHLB} which we use to organise the  $p$-adic Simpson correspondence is the Hodge--Tate stack of Bhatt--Lurie \cite[Construction~3.7]{bhatt2022prismatization}: Given any bounded $p$-adic formal scheme $X$, this is a functor $X^\HT$ fibred in groupoids on the category of $p$-nilpotent rings, defined in terms of ``generalized Cartier--Witt divisors''.  Instead of recalling the precise definition, which we will never need in this article, we now instead recall the key properties of $X^\HT$ which we used in \cite{AHLB} in order to define $p$-adic Simpson functors:

The functor $X^\HT$ is a stack for the fpqc-topology on $p$-nilpotent rings, and the association $X\mapsto X^\HT$ is functorial. There is a canonical morphism $X^\HT\to X$ which is natural in $X$. When $X$ is perfectoid, this map is an isomorphism.
In  \cite[Thm~1.2]{AHLB}, we used this to show that there is a canonical functor
\[\alpha^\ast_X: \mathcal{P}erf(X^\HT, \mathcal{O}\tf) \to \mathcal{P}erf(\X_v)\]
from the category of perfect complexes over the rational structure sheaf $\O\tf=\O\otimes_{\Z_p}\Q_p$ on $X^\HT$ to perfect complexes on $\X_v$. (In fact, in \cite{AHLB} there is for technical reasons a running assumption  that $\X$ is qcqs, but  the general case follows immediately from gluing, by naturality.) When $X$ is as in \Cref{setup:geometric} or \Cref{setup:arithmetic}, we showed that this functor is fully faithful \cite[Thm~1.2]{AHLB}. In particular, it restricts to a fully faithful functor
\begin{equation}\label{eq:alpha-ast-Bun(XHT)}
\alpha^\ast_X: \Bun(X^\HT,\O\tf)\to \Bun(\X_v).
\end{equation}
Here and in the following, for any ringed site $(\mathcal C,\O)$, we denote by $\Bun(\mathcal C,\O)$ the category of locally free $\O$-modules of finite rank on $\mathcal C$, and we drop $\O$ if it is clear from the context. Thus $\Bun(\X_v)=\Bun(\X_v,\O_{\X_v})$ is the category of v-vector bundles on $\X$.

\subsection{The local $p$-adic Simpson functor in the geometric setup}\label{s:recall-local-Simpson-geom}
Let now $X$ and $S$ be as in \Cref{setup:geometric}. Then the natural map $X^\HT\to X$ is a gerbe banded by $T_X^\sharp\{1\}$: Here $T_X\{1\}\to X$ is the tangent bundle of $X$ twisted by a Breuil-Kisin twist $-\otimes_{A/I}I/I^2$, and $T_X^\sharp\{1\}\to X$ is the formal group scheme defined as the PD-completion  of $T_X\{1\}$ at the identity section. Assume now first that there exists a splitting \[s:X\to X^\HT.\]
Such a splitting $s$ is induced by the datum of  a prismatic lift of $X$. For example, we will explain in \S\ref{sec:explicit-toric-charts} that a prismatic lift and hence a splitting is induced by a toric chart for $X$, which also appears in Faltings' local $p$-adic Simpson correspondence. More generally, it is given by the datum of any such splitting induces an isomorphism
\[ X^\HT\cong BT_X^\sharp\{1\}\]
where the right hand side denotes the classifying stack of $T_X^\sharp\{1\}$. By an easier special case of \cite[Thm~1.3]{AHLB}, pullback along the natural map $X\to BT_X^\sharp\{1\}$ induces an equivalence of categories
\begin{equation}\label{eq:constr-of-Higgs-field-via-Cariter-duality}
\Phi_X:\Bun(BT_X^\sharp\{1\})\isomarrow \left\{\begin{array}{c}\text{Higgs bundles $(M,\theta)$ on $X$  s.t.\ }\\ \text{$\theta$ is topologically nilpotent}\end{array}\right\}.
\end{equation}
Let us recall the construction of $\Phi_X$ from \cite[Thm~6.3, Remark~6.2]{AHLB}: The pullback of a vector bundle along $X\to BT_X^\sharp\{1\}$ is given by a vector bundle with a  $T_X^\sharp\{1\}$-action. On the level of $\O_X$-modules, this corresponds to an $\O_X$-module $M$ together with a co-action
\begin{equation}\label{eq:Cartier-duality-applied-to-Higgs}
M\to M\otimes \O(T_X^\sharp\{1\}).
\end{equation}
Let $\wtOm:=\Omega^1_{X|S}\{-1\}$, 
then via Cartier duality, \cite[Lemma~6.4]{AHLB}, this map corresponds to a morphism
\[
\widehat{\mathrm{Sym}^\bullet_X}(\wtOm^\vee)\otimes M\to M
\]
where the first term on the right is the completion of $\mathrm{Sym}^\bullet_X(\wtOm^\vee)$ at the ideal spanned by $\wtOm^\vee$. The Higgs field is now given by the $\O_X$-linear dual of the induced map $\wtOm^\vee \otimes M\to M$.

All in all, we thus obtain a commutative diagram of fully faithful functors
	\[
\begin{tikzcd}[row sep=0.1cm]
	& \Bun(X^{\HT},\O\tf)
	\arrow[ld, "\alpha_X^\ast"'] \arrow[rd, "\beta_{s}"] &                      \\
	\Bun(\mathcal X_v) \arrow[rr, "\text{$p$-adic Simpson}", dotted,<->] &                                                                    & \big\{\text{Higgs bundles on }\mathcal X\big\}
\end{tikzcd}\]
where $\beta_s$ is defined as the pullback along the isomorphism $BT_X^\sharp\{1\}\isomarrow X^\HT$ induced by $s$, followed by $\Phi_X$. By the above description of the essential image of $\beta_s$, this defines a fully faithful functor
\[\LS_s:\left\{\begin{array}{c}\text{Higgs bundles $(M,\theta)$ on $\X$  s.t.\ }\\ \text{$\theta$ is topologically nilpotent}\end{array}\right\}\hookrightarrow  \Bun(\mathcal X_v).\]
This is the local $p$-adic Simpson functor induced by a splitting of the Hodge--Tate stack (see \cite[Thm 6.29]{AHLB} for more details). Our first goal in this article in \S\ref{sec:explicit-toric-charts} below will be to compare this to Faltings' local correspondence in terms of toric charts.

\subsection{The global $p$-adic Simpson functor in the geometric setup}\label{s:recall-global-Simpson-geom}
Splittings of $X^\HT\to X$ only exist locally, and one can show that $X^\HT$ is rarely split as a gerbe when $X$ is for example proper. Inspired by the work of Faltings \cite{faltings2005p}, Abbes--Gros and Tsuji \cite{abbes2016p}, we therefore globalise the construction by comparing $X^\HT$ to the stack 
\[ \Lft_{X}\to X\]
of square-zero lifts of $X$ defined in \cite[Definition~7.5]{AHLB}. Like for $X^\HT$, we won't need the precise definition of $\Lft_{X}$ in this article. Instead, let us therefore only recall its key properties: The morphism $ \Lft_{X}\to X$ is natural in $X$, and it is a gerbe banded by the relative formal group scheme $T_X\{1\}$. This stack is  split by the datum of an $A_2$-lift $\tilde{X}$ of $X$.

More generally, it is more useful in practice to also allow a weaker datum of lifts, at the expense of a stronger convergence condition: Let $x\in S\tf$ be such that $S\subseteq xS$ and let $xS\{1\}\subseteq S\{1\}\tf$ be the image of $x\cdot:S\{1\}\to S\{1\}\tf$, then we define a square-zero thickening
\[ 0\to xS\{1\}\to A_2(x)\to S \to 0\]
as the pushout of $S\{1\}\to A_2\to S$ along $S\{1\}\to xS\{1\}$ (cf \cite[Definition~7.1]{AHLB}). Then in \cite[\S7.2]{AHLB} we more generally consider a stack
\[ \Lft_{X,x}\to X\]
which is now a gerbe banded by $xT_X\{1\}$, and the gerbe is split by the datum of an $A_2(x)$-lift $\tilde{X}$ of $X$, called an $x$-lift. More precisely, any such $x$-lift $\tilde{X}$ induces a section
\[ \rho_{\tilde X}:X\to \Lft_{X,x}.\]
 The point of this generalisation to $x$-lifts is that in contrast to prismatic lifts or $A_2$-lifts, such $x$-lifts typically also exists in global situations, e.g.\ for any proper smooth $X$.

\medskip

Let now $x\in S\tf$ be as before, set $\omega:=(1-\zeta_p)^{-1}$ and let $0\neq z\in S$ be such that $z\in S\tf^\times$ and $\omega x z \in S$. The standard setting is $x=1$ and $z=\zeta_p-1$, but it is beneficial to allow greater generality. 
Towards the global $p$-adic Simpson functor, the crucial point is now that by \cite[Proposition 7.8]{AHLB}, there are canonical and functorial morphisms of fpqc-stacks on $p$-nilpotent rings
\[ X^\HT\to \Lft_{X,x}\to z_{\ast}X^\HT\]
where $z_{\ast}X^\HT$ is the pushout of the gerbe $X^\HT\to X$ along the multiplication $z:T^\sharp_X\{1\}\to T^\sharp_X\{1\}$. Moreover, by \cite[Definition~7.9]{AHLB} there is a constant $u\in \Z_p^\times$ (conjecturally, $u=1$) such that these morphisms are linear over
\[ T^\sharp_X\{1\}\xrightarrow{ u\cdot \mathrm{can}} xT_X\{1\}\xrightarrow{\cdot u^{-1}z} T^\sharp_X\{1\}.\]
It follows from this that the stack $z_{\ast}X^\HT$ is a pushout of $\Lft_{X,x}$, and is thus also split by the datum of an $A_2(x)$-lift $\tilde{X}$. In particular, the datum of such a lift $\tilde{X}$ induces an isomorphism
\[\Phi_{\tilde X}:z_{\ast}X^\HT=BT^\sharp_X\{1\}.\]

Exactly like before, we can now use $\Phi_{\tilde X}$ to obtain first a diagram of fully faithful functors 

\begin{equation}\label{eq:def-alpha-beta}
\begin{tikzcd}[row sep=0.1cm]
	& \Bun(z_{\ast}X^\HT,\O\tf)
	\arrow[ld, "\alpha_X^\ast"'] \arrow[rd, "\beta_{\tilde{X}}"] &                      \\
	\Bun(\mathcal X_v) \arrow[rr, "\text{$p$-adic Simpson}", dotted,<->] &                                                                    & \big\{\text{Higgs bundles on }\mathcal X\big\}.
\end{tikzcd}
\end{equation}
In a second step, like in the local case discussed in \S\ref{s:recall-local-Simpson-geom}, we deduce from the description of the essential image of $\beta_{\tilde{X}}$ a fully faithful functor (see \cite[Thm 7.13]{AHLB})
\[ \mathrm{S}_{\tilde{X},z}:\left\{\begin{array}{c}\text{Higgs bundles $(M,\theta)$ on $\X$  s.t.\ }\\ \text{$\frac{1}{z}\theta$ is topologically nilpotent}\end{array}\right\}\hookrightarrow  \Bun(\mathcal X_v).\]
This functor can be computed using a period sheaf $\mathcal B_{\tilde{X}}$ that arises from the geometry of the Hodge--Tate stack: For simplicity, let $x=1$ and $z=\zeta_p-1$, then there is a Higgs module $(\mathcal B_{\tilde{X}},\Theta_{\mathcal B_{\tilde{X}}})$ on $\X_v$ (see \cite[Definition~17.7]{AHLB} for the precise definition) such that for any Higgs bundle $(M,\theta)$ on $\X$ with topologically nilpotent $\frac{1}{z}\theta$, the associated v-vector bundle is described as follows: Let $\nu:\X_v\to \X_\an$ be the natural map, then
\begin{equation}\label{eq:compute-StX-using-period-sheaf}
\mathrm{S}_{\tilde{X}}(M,\theta) \cong \ker(\mathcal{B}_{\tilde{X}} \otimes_{\mathcal{O}_{\X}}  \nu^\ast M \xrightarrow{\Theta_{\mathcal{B}_{\tilde{X}}} \otimes \mathrm{Id} + \mathrm{Id} \otimes \theta_M}  \mathcal{B}_{\tilde{X}} \otimes_{\mathcal{O}_{\X}} \nu^\ast M \otimes_{\mathcal{O}_{\X}} \nu^\ast \Omega_{\X}^1\{-1\})
\end{equation}
	is naturally isomorphic. Since $\mathcal{B}_{\tilde{X}}$ can be described explicitly, this gives a way to compute $\mathrm{S}_{\tilde{X}}$.

	The first main goal of this article is to describe the essential image of  $\mathrm{S}_{\tilde{X}}$, leading to \Cref{t:relative-small-correspondence}.
	
\subsection{The local $p$-adic Simpson functor in the arithmetic setup}\label{s:recall-local-Simpson-arith}
Let now $X$ be a smooth formal scheme over $\O_K$ as in \Cref{setup:arithmetic}. We recall some results from \cite[\S6.4]{AHLB}, specialised to the case of vector bundles: In the arithmetic setup, the Hodge--Tate stack $X^\HT \to X$ is a gerbe banded by a group scheme $G\to X$ that can be described as follows: It is a semi-direct product
\[ 0\to T^\sharp_{X}\{1\}\to G\to H\to 0\]
where $H=\mathbb G_m^\sharp$  or $H=\mathbb G_a^\sharp$ depending on whether $K$ is absolutely unramified or not. Here for any $p$-adic formal group over $X$, we denote by $-^\sharp$ the PD-completion at the identity section. If there is a global splitting $X\to X^\HT$, it follows that this splitting induces an isomorphism
\[X^\HT\cong BG.\]
The motivation for \Cref{d:Higgs--Sen} is now that the category of Higgs--Sen modules on $\X$ is canonically equivalent to the category of finite locally free $\O\tf$-modules on $BG$: Roughly speaking, the Higgs field $\theta$ comes from the normal subgroup $T^\sharp_X\{1\}$, the Sen operator $\phi$ comes from the subgroup $H$, and the compatibility between $\theta$ and $\phi$  reflects the semi-direct product structure. In particular, exactly as in the geometric case, any such local splitting induces a natural fully faithful functor
\[\mathrm{S}_{\pi}: \left\{\text{Hitchin-small Higgs--Sen bundles on $\X$}\right\}\hookrightarrow  \Bun(\mathcal X_v).\]
\subsection{The global $p$-adic Simpson functor in the arithmetic setup}\label{s:recall-global-Simpson-arith}
Finally, there is (only in the arithmetic case!) a canonical way to glue the local $p$-adic Simpson functors to a global correspondence which only depends on the choice of uniformiser $\pi$ of $\O_K$. We recall the main ideas and refer to  \cite[\S7.4]{AHLB} for details: The construction is achieved by combining the local functors with the geometric $p$-adic Simpson functor of $X_{\O_C}$ from \S\ref{s:recall-global-Simpson-geom}, using the canonical isomorphism
\[ X_{\O_C}^\HT=X^\HT\times_{\Spf(\O_K)^\HT}\Spf(\O_C).\]
Here we can employ the $p$-adic Simpson correspondence of \Cref{setup:geometric} applied to $X_{\O_C}$:
If $K$ is unramified, then there is a natural map $\O_K\to A_2$ which we can use to give a canonical lift $\tilde{X}_{\O_C}$ of $X_{\O_C}$ to $A_2$. The $p$-adic Simpson correspondence $S_{\tilde{X}_{\O_C}}$ over $\O_C$ is then naturally compatible with the functor $S_{\pi}$ in the arithmetic setting via the functor that sends a Higgs--Sen bundle $(N,\theta,\phi)$ to the base-change $(E_{\O_C},\theta_{\O_C})$ of the underlying Higgs bundle. 

In general, when $K$ is allowed to be ramified, there is only a map $\O_K\to A_2(x)$. This was one motivation to develop the global correspondence of \S\ref{s:recall-global-Simpson-geom} for ``$x$-lifts'' of $X$ to $A_2(x)$. This allows us to still obtain a canonical small $p$-adic Simpson correspondence for $X_{\O_C}$ as in \S\ref{s:recall-global-Simpson-geom}. Similar to Min--Wang's result \cite[Thm~6.4]{MinWang22}, we arrive at a fully faithful global arithmetic  $p$-adic Simpson functor
\[ \mathrm{S}_{\pi}:\{\text{Hitchin-small Higgs--Sen bundles on }\X\}\hookrightarrow  \Bun(\mathcal X_v)\]
by Galois descent from the geometric case. The second main goal of this article is to describe the essential image of this functor in \S\ref{sec:essential-image-arithmetic}.
\section{Essential Image: geometric case}

Throughout this section, we work in the geometric setting of \Cref{setup:geometric}, so $S$ is a perfectoid formal scheme and $X\to S$ is smooth morphism. Let $\X\to S^\ad_\eta$ be its adic generic fibre.
 \subsection{Explicit description of $S_{\tilde{X}}$ in terms of toric charts}\label{sec:explicit-toric-charts}
	We begin by comparing the functor $\mathrm{S}_{\tilde{X}}$ from \S\ref{s:recall-global-Simpson-geom} to the local correspondence $\mathrm{LS}^{\mathrm F}_{c}$ of Faltings \cite{faltings2005p} defined in terms of a toric chart $c$: This was constructed in \cite[p849]{faltings2005p} for smooth formal schemes\footnote{More precisely, in \cite{faltings2005p} Faltings works in an algebraic setting, but the construction generalises verbatim to the setting of formal schemes. Moreover, \cite{faltings2005p} works in a more general setting of semi-stable schemes.} over $\O_C$. We here use the smoothoid generalisation from \cite{heuer-sheafified-paCS}, which we now recall. We begin by defining global notions of smallness.
	
	Throughout this subsection, we assume that $S$ contains $\Z_p^\cycl$. For any $\alpha\in \tfrac{1}{p-1}\Z\tf_{\geq 0}$, we then denote by $p^\alpha$ any element in  $\Z_p^\cycl$ whose absolute value is equal to $|p|^\alpha$. 

    \begin{definition}
    	\label{def:faltings-small-v-bundles}
    	\begin{enumerate}
    		\item A Higgs bundle $(M,\theta_M)$ on $\X$ is called \textit{Faltings-small} if Zariski-locally on $X$, there exists a finite free integral model $\mathfrak M$ of $M$ on $X$ such that $\theta_M$ restrict to a Higgs field $\theta_{\mathfrak M}:\mathfrak M\to \mathfrak M\otimes \Omega^1_X(-1)$ which is trivial modulo $p^{2\alpha}$ for some $\alpha>\frac{1}{p-1}$. We denote the category of Faltings-small Higgs bundles on $\X$ by $\Higgs^{\Fsm}(\X)$.
    		\item 
    	A v-vector bundle $V$ on $\X$ is called \textit{Faltings-small} if Zariski-locally on $X$, it admits a reduction of structure group to $1+p^{2\alpha}M_n(\O_{\X_v}^+)\subseteq \GL_n(\O_{\X_v})$ for some $\alpha>\frac{1}{p-1}$. We denote the category of Faltings--small v-vector bundles on $\X$ by $\mathrm{Bun}^{\Fsm}(\X_v)$.
    	\end{enumerate}
    \end{definition}
    \begin{definition}\label{d:Hitchin-small}
    	Let $z\in S\cap S\tf^\times$.
    	If $\X$ is a rigid space, i.e.\ if $S\tf$ is a perfectoid field, then a Higgs bundle $(M,\theta_M)$ on $\X$ is called $z$-\textit{Hitchin-small} if for any $\delta\in \frac{1}{z}\Omega^\vee_X\{1\}$, the endomorphism  $\delta\circ \theta_{M}:M\to M$ is topologically nilpotent with respect to the canonical topological structure on the coherent $\O_\X$-module $\End(M)$. For general perfectoid bases $S$, a Higgs bundle $(M,\theta_M)$ is called $z$-Hitchin-small if for any geometric point $\xi:\Spa(C,C^+)\to \Spa(S\tf,S)$, the pullback of $(M,\theta_M)$ to the rigid space $\mathcal X\times_{\Spa(S[\frac{1}{p}],S)}\Spa(C,C^+)\to \mathcal X$ is $z$-Hitchin-small. We denote by $\Higgs^{\zHsm}(\X)$  the category of $z$-Hitchin-small Higgs bundles on $\X$. When we simply say ``Hitchin small'' without specifying $z$, this refers to the case of $z=(1-\zeta_p)$.
    \end{definition}
	\begin{remark}
		Note that in comparison to \S\ref{s:recall-global-Simpson-geom}, the factor in \Cref{def:faltings-small-v-bundles} is $p^{2\alpha}$ rather than $p^{\alpha}$. The reason for this difference is essentially the difference between the Breuil--Kisin--Fargues twist $\Omega^1_X\{-1\}$ and the Tate twist $\Omega^1_X(-1)$. If we used a Tate twist $(-1)$ instead of $\{-1\}$ in \S\ref{s:recall-global-Simpson-geom}, it would also be the same factor by  $p^{2\alpha}$. One reason why we use $\{-1\}$ in \S\ref{s:recall-global-Simpson-geom} is that this makes sense over any perfectoid base, whereas $(-1)$ is only defined when $\Z_p^\cycl$ is contained in $S$.
	\end{remark}

It is clear that this definition of Hitchin-smallness is equivalent to the more geometric definition in \S\ref{s:intro-small-paS-moduli} in terms of the Hitchin fibration, whence the name. Moreover, we clearly have:
\begin{lemma}\label{l:Faltings-implies-Hitchin-small}
	Any Faltings-small Higgs bundle is automatically Hitchin-small.
\end{lemma}

We now recall the statement of the local $p$-adic Simpson correspondence in the smoothoid setting. For this we furthermore need the following definition:

\begin{definition}
	A toric chart is an affine \'etale morphism of formal schemes \[c:X\to \mathbb T^d_S\] for some $d\in \N$, where  $\mathbb T^d_S$ is the $d$-dimensional affine formal torus over $\Spf(S)$. We call a $p$-adic formal scheme $X$ over $S$ smoothoid if Zariski-locally, it admits a toric chart.
\end{definition}

Throughout this subsection, we assume that $X$ is a smoothoid formal scheme over $S$  and that $X$ admits a toric chart $c:X\to \mathbb T^d_S$. For any $n\in\N$, we denote by $X_n\to X$
the pullback of the map $[p^n]: \mathbb T^d_S\to  \mathbb T^d_S$ along $c$, and by \[\X_n\to \X\] its adic generic fibre, which is a finite \'etale $\mu_{p^n}^d$-torsor. The inverse limit $\mathcal X_\infty:=\varprojlim \X_n$ in diamonds is represented by an affinoid perfectoid space, and \[\X_\infty\to \X\] is a pro-finite-\'etale Galois torsor under the profinite group $\Delta:=\Z_p(1)^d=\varprojlim_n \mu_{p^n}^d$.

The following is a  more precise version of \cite[Lemma~6.4]{heuer-sheafified-paCS} in our situation of good reduction:
\begin{lemma}\label{l:locally-Faltings-small}
	Assume that $X$ is a smoothoid adic space that admits a toric chart $c:X\to \mathbb T^d_S$. 
	\begin{enumerate}
		\item Any Higgs bundle $(M,\theta_M)$ on $\X$ becomes Faltings-small on $\X_n$ for $n\gg 0$.
		\item 
	Any v-vector bundle $V$ on $\X$ becomes Faltings-small on $\X_n$ for $n\gg 0$.
	\end{enumerate}
\end{lemma}
\begin{proof}
	
	For part (1),  it is clear from Kiehl's Theorem that there exists a Zariski-cover of $X$ which trivialises $M$. As the statement is local, we may therefore assume that $M$ is trivial. We may then choose a free model $\mathfrak M$ and  it suffices to see that with respect to this model,  after pullback along the formal model $X_n\to X$, the Higgs field $\theta_M$ restricts to $\mathfrak M$ and becomes divisible by $p^{2\alpha}$.
	By explicit calculation on $[p^n]:\mathbb T^d\to \mathbb T^d$ and \'etale base-change, we see that the image of $\Omega_X^1$ in $\Omega^1_{X_n}$ is  $p^{n}\Omega^1_{X_n}$, hence pullback along $X_n\to X$ rescales $\theta_M$ by $p^n$. Hence $\theta_M$ becomes small for $n\gg 0$.
	
	For part (2), the pullback $V_\infty$ of the v-vector bundle $V$ to the perfectoid cover $\X_\infty\to \X$ is an analytic vector bundle. By \cite[Corollary~5.4.42]{gabber2003almost}, this is isomorphic to the pullback of some analytic vector bundle on $\X_m$ for some $m\gg0$. After replacing $X_m$ by a Zariski-open cover, we can assume that this vector bundle is trivial, and hence so is $V_\infty$. Fix a trivialisation $V_\infty(\X_\infty)=\O(\X_\infty)^n$, then the $\Delta_m:=\Gal(\X_\infty|\X_m)$-action on $M_\infty:=V_\infty(\X_\infty)=\O(\X_\infty)^n$ defines a continuous $1$-cocycle $\Delta_m\to  \GL_n(\O(\X_\infty))$. After increasing $m$, hence shrinking $\Delta_m$, this factors through a map $\Delta_m\to  1+p^{2\alpha}\mathfrak m M_n(\O^+(\X_\infty))$ which defines the desired reduction of structure group.
\end{proof}

    Assume that $X=\Spf(R)$ is affine with a given toric chart $c:X\to \mathbb T^d_S$ over some perfectoid base $S$.    As before,
   this induces a perfectoid cover $X_\infty=\Spf(R_\infty)\to X$ whose adic generic fibre is pro-finite-\'etale Galois with group $\Delta\cong \Z_p(1)^d$. Moreover, the toric chart induces  a canonical isomorphism $\Omega_{R}^1(-1)\tf\isomarrow \Hom_{\Z_p}(\Delta,R\tf)$ which we can dualise to a $\Z_p$-linear homomorphism
   \begin{equation}\label{d:def-D}
    D_c:\Delta\to \Hom_R(\Omega^1_{R}(-1),R).
    \end{equation}
    \begin{theorem}[{\cite[Thm~6.5]{heuer-sheafified-paCS}}]\label{t:LSF-is-equivalence}
    	The toric chart $c$ induces an equivalence of categories 
    	\[ \LS^{\mathrm F}_{c}:    \Higgs^{\Fsm}(\X)  \isomarrow  \mathrm{Bun}^{\Fsm}(\X_v)  \]
    	given by sending a Higgs bundle $(M,\theta_M)$ on $\X$ to the v-vector bundle $V$ on $\X$ defined as follows: Consider the $R_\infty$-module  $M\otimes_R R_\infty$ with semi-linear $\Delta$-action defined by letting any $\sigma\in \Delta$ act as
    	\[ \exp(D_c(\sigma)\circ \theta_M)\otimes \sigma:M\otimes_RR_\infty\to M\otimes_RR_\infty.\]
    	Interpreting this as a descent datum along $ \X_\infty\to \X$ defines the v-vector bundle $\LS_c^{\mathrm F}(M,\theta)$ on $\X$.
	\end{theorem}
\begin{proof}
	In order to see that this follows from {\cite[Thm~6.5]{heuer-sheafified-paCS}}, we only need to compare the notions of ``smallness'', as the one used in \cite[Definition~6.2]{heuer-sheafified-paCS} differs from the one in our
	\Cref{def:faltings-small-v-bundles}\footnote{
	In \cite{faltings2005p}, Faltings uses both of these variants of ``smallness'', the global and the local one.}. To clarify the relation, we first note that in our case of good reduction, we can take $c=\frac{2}{p-1}$ in \cite[Definition~6.2]{heuer-sheafified-paCS}, so ``small'' means here that a v-vector bundle globally on $X$ admits a reduction of structure group to $1+p^{2\alpha}M_n(\O^+)$ for some $\alpha>\frac{1}{p-1}$. Similarly for Higgs bundles. But the naturality in the toric chart in \cite[Thm~6.5.3]{heuer-sheafified-paCS} guarantees that the correspondence glues over Zariski-covers of $X$. Here we can use the same $c$ as Zariski-localisation preserves the good reduction property. Hence we obtain the claimed correspondence for Faltings-small bundles.
\end{proof}

Our goal in this subsection is to compare $\LS^{\mathrm F}_{c}$ to our functor $\mathrm{S}_{\tilde{X}}$. To make this precise, we first explain how the toric chart $c:X\to \mathbb T^d_S$ induces a prismatic lift $(B,J)$ of $R$:

 The prismatic lift $(A,I)$ of $S$ induces a canonical lift $A\langle T_1^{\pm 1},\dots,T_d^{\pm 1}\rangle$ of $\mathbb T^d_S$. This has a canonical $\delta$-structure extending the one on $A$, defined by $\delta(T_i)=0$ for $i=1,\dots,d$. Since $A$ is $I$-adically complete, we have $A_{\et}=R_\et$. So the chart $S\langle T_1^{\pm 1},\dots,T_d^{\pm 1}\rangle\to R$ lifts to an $I$-adically \'etale map  \[A\langle T_1^{\pm 1},\dots,T_d^{\pm 1}\rangle\to B\] in an essentially unique way. By \cite[Lemma~2.18]{Bhatta}, the delta-structure extends to $B$ in a unique way. Set $J:=IB$, then $(B,J)$ is the desired prismatic lift of $R$. In particular, we obtain an $A_2$-lift \[\tilde{X}:=\Spf(B/J^2).\]

Second,  the natural map $R\to R_\infty$ admits a lift to a morphism $\varphi_c:A\to A_{\inf}(R_\infty)$: For this we choose  compatible systems of $p$-power roots $T_i^\flat=(T_i,T^{1/p}_i,\ldots )\in R_{\infty}^\flat$ of the images of the $T_i$ in $R_\infty$. Like in \cite[\S4.2]{AHLB}, the canonical lift $A\to A_{\inf}(R_\infty)$ of $S\to R_\infty$ then extends to a map
\[A\langle T_1^{\pm 1},\dots,T_d^{\pm 1}\rangle\to A_{\inf}(R_\infty),\quad T_i\mapsto [T_i^\flat].\]
By \'etale lifting, we can now find a unique lift of the extension $R\to R_\infty$ of $S\langle T_1^{\pm 1},\dots,T_d^{\pm 1}\rangle \to R_\infty$ to the desired map $B\to A_{\inf}(R_\infty)$. Quotienting by $I^2$, this induces a morphism $\O(\tilde{X})\to A_2(R_\infty)$.

We recall from \cite[Remark~7.26]{AHLB} that we can use this to make the period ring $\mathcal B_{\wt X}$ from \S\ref{s:recall-global-Simpson-geom} computing $\mathrm{S}_{\wt X}$  more explicit:
Fix generators $y_1,\dots,y_d$ of $\Omega^1_R\{-1\}$ as an $R$-module. For example, when we fix a compatible choice of $p$-power unit roots to trivialise the Breuil--Kisin--Fargues twist, we can use the pullbacks of the canonical generators $\frac{dT_1}{T_1},\dots,\frac{dT_d}{T_d}$ of $\Omega^1_{\mathbb T^d_S}$ along $c$.
\begin{lemma}\label{l:explicit-description-of-B}
	There are $Y_1,\dots,Y_d\in H^0(\X_\infty,\mathcal B_{\wt X})$ such that the chart  $c$ induces an isomorphism
	\[ H^0(\X_\infty,\mathcal B_{\wt X})\cong R_\infty\big[\tfrac{(1-\zeta_p)^nY_i^n}{n!}, n\in \N,i=1,\dots,d\big]^\wedge_p\tf\]
	with respect to which we have $\Theta_{\mathcal B_{\wt X}}=\sum \frac{\partial}{\partial Y_i}y_i$.
\end{lemma}
\begin{proof}
	By \cite[Remark~7.26]{AHLB} this holds for $X=\Spf(R)=\mathbb T^d_S$. The general case follows from this: The natural map $X^\HT\to (\mathbb T^d_S)^\HT\times_{\mathbb T^d_S}X$ is an isomorphism since $c:X\to \mathbb T^d_S$ is \'etale. Hence it follows from the definition in \cite[Definition~17.7]{AHLB} that $\mathcal B_{\wt X}$ is the base-change of the period ring for $\mathbb T^d_S$ along $c$.
\end{proof}

        \begin{proposition}\label{p:comp-LS_F-LS_prism}
    	There is a canonical transformation $\gamma$ making the following diagram 2-commutative:
    	\[
    	\begin{tikzcd}[row sep =0.3cm,column sep =0.3cm]
    		\Higgs^{\Fsm}(\X) \arrow[dd, "\LS^{F}_{c}","\sim"'] \arrow[rr, hook] &  & \Higgs^{\Hsm}(\X) \arrow[dd, "\mathrm{S}_{\tilde{X}}"]\arrow[lldd, Rightarrow,"\gamma",shorten >=0.5cm,shorten <=0.5cm] &                             \\
    		&  &\\
    		\mathrm{Bun}^{\Fsm}(\X_v) \arrow[rr, hook]                          &  & \mathrm{Bun}(\X_v)       &
    	\end{tikzcd}\]
    \end{proposition}
    \begin{proof}
    	The inclusion on top exists by \Cref{l:Faltings-implies-Hitchin-small}.
    	Let $\mathcal{M}=(M,\theta_M)$ be any Faltings-small Higgs bundle on $\X$. Let
    	\[V:=\mathrm{S}_{\tilde{X}}(\mathcal M)\]
    	be the associated v-vector bundle on $\X$. We recall from \S\ref{s:recall-global-Simpson-geom} that $\mathrm{S}_{\tilde X}$ can be described using the period sheaf $(\mathcal{B}_{\tilde{X}}, \Theta_{\mathcal{B}_{\tilde{X}}})$: Let $B:=H^0(\X_\infty,\mathcal{B}_{\tilde{X}})$, then by  \eqref{eq:compute-StX-using-period-sheaf} we have a canonical identification
    	\[V(\X_\infty)= \ker(M\otimes_R B\to M\otimes_R B\otimes \Omega^1_{R}(-1)).\] As a first step, we use this to define a natural morphism
    	\[ \phi:M\otimes_RR_\infty \to M\otimes_RB \]
    	as follows: Set $\omega:=(1-\zeta_p)^{-1}$. By \Cref{l:Faltings-implies-Hitchin-small}, the Higgs field $\omega\cdot \theta_M:M\to M\otimes_R \Omega_R^1(-1)$ is topologically nilpotent.  Therefore $\theta$ determines a morphism
    	\[\widehat{\mathrm{Sym}}^{\bullet}_{R}(\omega\cdot \Omega_R^1(-1)^\vee)\otimes_R M\to M\]
    	where the completion on the left is with respect to the zero section. Like in \eqref{eq:Cartier-duality-applied-to-Higgs}, this corresponds via Cartier duality, \cite[Lemma 6.4]{AHLB}, to a morphism
    	\[\phi_0:M\to M\otimes_R B_0\]
    	where $B_0:=R\big[\tfrac{(\omega\cdot Y_i)^n}{n!}, n\in \N,i=1,\dots,d\big]^\wedge_p\tf$.
    	More precisely, using the formula defining the Cartier duality \cite[(9) in the proof of Lemma 6.4]{AHLB}, we see that $\phi_0$ can be described as follows: Using the chosen generators $y_1,\dots,y_d$, write $\theta_M=\sum \theta_iy_i$ for some $\theta_1,\dots,\theta_d\in \mathrm{End}(M)$, then
    	\begin{equation}\label{eq:expl-descr-of-phi}
    		\phi_0(m)= \prod_{i=1}^d\sum_{n=0}^\infty \frac{Y_i^n}{n!}\theta_i^n(m)=\prod_{i=1}^d\exp(\theta_iY_i)(m).
    	\end{equation}
    	Here we note that the factors commute since the $\theta_i$ commute due to the Higgs field condition.
    	
    	By the explicit description in \Cref{l:explicit-description-of-B}, we have a natural map $B_0\hookrightarrow B$. 
    	We can therefore now extend $\phi$ in an $R_\infty$-linear way to obtain the desired morphism $\phi$.
    	For this we have:
    	\begin{lemma}\label{l:comp-Higgs-v-bundle-on-wtX}
    		The image of the morphism $\phi$ agrees with $V(\X_\infty)$. Moreover, the following diagram commutes for any $\sigma\in \Delta$:
    		    	\[
    		\begin{tikzcd}
    			M\otimes_RR_\infty \arrow[d, "\exp(D_c(\gamma)\circ\theta_M)\otimes \sigma"'] \arrow[r,"\phi","\sim"'] & V(\X_\infty) \arrow[d, "\gamma"] \\
    			M\otimes_RR_\infty \arrow[r,"\phi","\sim"']                                   & V(\X_\infty).             
    		\end{tikzcd}\]
    	\end{lemma}
    	\begin{proof}
    		It is clear from \eqref{eq:expl-descr-of-phi} that $\phi$ is injective: Indeed the $R_\infty$-linear section $B\to R_\infty\tf$ given by projection to the constant coefficient composes with $\phi$ to the inclusion $M=M\otimes_RR\subseteq M\otimes_RR_\infty\tf$.
    		
    		For the first claim, it thus suffices by \cite[Lemma 7.18]{AHLB} to see that the the image of $\phi$ is in the kernel of $\theta_M\otimes 1+1\otimes \Theta:M\otimes B\to M\otimes B\otimes \Omega^1_R(-1)$. This follows from the computation
    		\[ (\theta_M\otimes 1+1\otimes \Theta)(\prod_{i=1}^d\exp(\theta_iY_i))=\sum_{j=1}^d\Big(\theta_j\prod_{i=1}^d\exp(\theta_iY_i)-\frac{\partial}{\partial Y_j}\prod_{i=1}^d\exp(\theta_iY_i)\Big) y_j=0\]
    		because $\frac{\partial}{\partial Y_i}\exp(\theta_iY_i)=\theta_i\exp(\theta_iY_i)$.
    		
    		It remains to see that the diagram commutes, for which it suffices to compare the two diagonal compositions on $M\subseteq M\otimes R_\infty$. As before, write $\theta_M=\sum \theta_iy_i$ for some $\theta_i\in \mathrm{End}(M)$, as well as
    		$D(\sigma)=\sum e_iy_i^\vee$
    		for some $e_i\in R_\infty$ in terms of the dual basis $y^\vee$ of $\Hom(\Omega^1_R\{-1\},R)$. Then
    		\[D(\sigma)\circ\theta_M=\sum_{i=1}^d e_i\theta_i\in \mathrm{End}(M).\]
    		
    		Let $m\in M$. Going first down and then to the right we calculate that
    		\begin{align*}
    			\phi(\exp(D(\sigma)\circ\theta_M)(m))=\phi(\prod_{i=1}^d\exp(e_i\theta_i))&=\prod_{i=1}^d\exp(\theta_iY_i)\exp(e_i\theta_i)
    			\intertext{whereas going first to the right and then down we obtain}
    			(\Id\otimes \sigma)(\phi(m))=(\Id\otimes \sigma)(\prod_{i=1}^d\exp(\theta_iY_i))(m)
    			&=\prod_{i=1}^d\exp(\theta_i(Y_i+e_i))(m)
    		\end{align*}
    		where in the last step we use that the action of $\sigma$ on $B$ is given by translation by $D(\sigma)$, which sends $Y_i\mapsto Y_i+e_i$ by definition of $e_i$. These two terms agree by commutativity of the $\theta_i$.
    	\end{proof}
    	By definition of $\mathrm{LS}^{\mathrm F}_c$, it follows from \Cref{l:comp-Higgs-v-bundle-on-wtX} that $\phi$ defines an isomorphism 
    	\[\phi:\LS^{\mathrm F}_c(\mathcal{M})\isomarrow V= \mathrm{S}_{\tilde{X}}(\mathcal{M}).\]
    	Since $\phi$ is clearly natural in $\mathcal{M}$, this finishes the proof of \Cref{p:comp-LS_F-LS_prism}.
    \end{proof}
 \begin{remark}
 	We note that in the proof of \Cref{p:comp-LS_F-LS_prism}, we did not use the full generality of the notion of Hitchin smallness, since we only required a factor of $p^{\frac{1}{p-1}}$, not of $p^{\frac{2}{p-1}}$. Indeed, the argument shows more generally that $\LS_c^{\mathrm{F}}$ is compatible with the local correspondence for ``$\omega$-Hitchin-small Higgs bundles'' of \cite[Theorem~6.37]{AHLB}. Indeed,  it is only in the globalisation step that the stronger notion of Hitchin-small Higgs bundles is needed in an essential way. This mirrors a similar phenomenon in Faltings' work.  To simplify the exposition, in this article, we ignore the difference in convergence conditions for the local versus the global correspondence. We therefore only use the weaker notion of Hitchin-smallness required by the global correspondence.
 \end{remark} 
   
\begin{corollary}\label{c:comp-of-LS-prism-with-HTlog}
	The local $p$-adic Simpson correspondence is compatible with the sheafified correspondence $\HTlog$ of  \cite[Thm~1.2]{heuer-sheafified-paCS} in the sense that the following diagram commutes
	\[
	\begin{tikzcd}
		\Higgs^{\Hsm}(\X) \arrow[d, "\mathrm{S}_{\tilde{X}}"]\arrow[r] & (M_n(\O)\otimes \Omega^1_{\mathcal X}(-1))\sslash\GL_n(\X)\\
		\mathrm{Bun}(\X_v) \arrow[r]& R^1\nu_{\ast}\GL_n(\X)\arrow[u,"\HTlog","\sim"']
	\end{tikzcd}
	\]
	where the categories on the right are the sets of isomorphism classes of Higgs bundles and v-vector bundles on $\X$ up to \'etale sheafification, respectively, and where the horizontal morphism are given by sending an object to its isomorphism class.
\end{corollary}
\begin{proof}
	Since $\HTlog$ is a morphism of sheaves on $\X_{\et}$, the statement is Zariski-local on $X$, so we may assume that $X$ admits a toric chart $c$ and $\tilde{X}$ is the lift induced by $c$.  Let $(M,\theta_M)$ be a Higgs bundle on $X$.  Then by \Cref{l:locally-Faltings-small}, the pullback of $(M,\theta_M)$ along the toric cover $f_n:\X_n\to \X$ induced by the toric chart becomes Faltings-small for $n\gg0$. Note that $f_n$ admits a formal model $X_n\to X$ which admits a lift $\tilde{f}_n:\tilde{X}_n\to \tilde{X}$ induced by the toric chart. By naturality of $\mathrm{S}_{\tilde{X}}$ in $\tilde{X}$, we can now replace $X$ by $X_n$ and thus reduce to the case that $(M,\theta_M)$ is Faltings-small.
	
	In this case, $\HTlog^{-1}$ is by definition in \cite[\S5]{heuer-sheafified-paCS} given by $\LS_c^{\mathrm F}(M,\theta_M)$. But by \Cref{p:comp-LS_F-LS_prism}, the isomorphism class of $\LS_c^{\mathrm F}(M,\theta_M)$ in $H^1_v(\X,\GL_n)$ agrees with that of $\mathrm{S}_{\tilde{X}}(M,\theta_M)$.
\end{proof}
\begin{remark}
	We note that in \cite[p850-851]{faltings2005p}, Faltings also gives a second version of the local $p$-adic Simpson correspondence for toric  $X$ over $\O_C$ defined in terms of a lift $\tilde{X}$ to $A_\inf(\O_C)/\xi^2$ (cf \S\ref{sec:relat-prev-work}). Namely, in terms of the notation used in this subsection, Faltings chooses a lift
	\[ \varphi:A/\xi^2\to A_\inf(R_\infty)/\xi^2\]
	of $A\to A_\inf(R_\infty)$. This induces for any $\sigma\in\Delta$ a derivation
	\[ A\to R_\infty,\quad a\mapsto \theta_M(\tfrac{\sigma \varphi(a)-\varphi(a)}{\xi})\]
	which factors through the reduction $A\to R$ and thus defines a $1$-cocycle
	$D:\Delta\to \Hom(\Omega^1_R,R_\infty)$.
	We observe that for any $\sigma\in \Delta$, the element $D(\sigma)$ agrees by definition with the derivation $D$ from \cite[Thm 3.13]{AHLB} computed in \cite[Lemma 4.7]{AHLB}, hence the notation. 
	
	Let now $(M,\theta_M)$ be any Faltings-small Higgs bundles on $X$. Then Faltings' second correspondence is defined by considering the $R_\infty\tf$-module $M\otimes R_\infty$ with semilinear $\Delta$-action defined by letting $\sigma\in \Delta$ act as
	\[\exp(D(\sigma)\circ \theta_M):M\to M\otimes R_\infty, \]
	and extended $R_\infty$-semilinearly. This defines a descent datum for a vector bundle along $\mathcal X_\infty\to \mathcal X$, and thus a v-vector bundle $\LS^{\mathrm F}_{\tilde{X}}(M,\theta_M)$ on $\mathcal X$. Exactly as in \Cref{p:comp-LS_F-LS_prism}, one can construct a morphism $\phi$ that describes a natural transformation $\LS^{\mathrm F}_{\tilde{X}}\Rightarrow \mathrm{S}_{\tilde{X}}$ to our functor. This reproves in particular that Faltings construction is independent of the auxiliary choice of $R_\infty$ and $\varphi$.
\end{remark}

Lastly, we need the canonical Higgs field on $V$ first described by Rodr\'iguez Camargo \cite{camargo2022geometric}: In the setting of \Cref{t:LSF-is-equivalence}, it is immediate from the definition that the $\Delta$-action on $M\otimes_RR_\infty$ commutes with $\theta_M\otimes \id$. Hence $\theta_M$ descends to a Higgs field $\theta:V\to V\otimes \Omega^1_X(-1)$ on $V:=\LS_c^{\mathrm F}(M,\theta)$. One can show that for a given v-vector bundle, this Higgs field is independent of the chart $c$:
\begin{theorem}[{\cite[Thm.~4.8]{heuer-proper-correspondence},\cite[Thm.~3.2.1]{HX}}]\label{t:canonical-HF}
	Let $Z$ be any smoothoid adic space and let $V$ be a v-vector bundle on $Z$. Then there is a natural Higgs field $\theta_V:V\to V\otimes \Omega_Z^1$ such that \'etale-locally on any toric space $\X\to Z$ where $V$ becomes Faltings-small and any toric chart $c$ for $\X$, the induced map $\theta_V(\X_\infty):V(\X_\infty)\to V(\X_\infty)\otimes \Omega_X^1$ is given by $\theta_M$ where $(M,\theta_M)=(\LS_c^F)^{-1}(V)$.
\end{theorem}

\subsection{The essential image in v-vector bundles in the geometric case}\label{sec:essential-image-vVB-geom-case}
Let now $X$ be any smoothoid $p$-adic formal scheme with rigid generic fibre $\X$ and let $\tilde{X}$ be a lift of $X$.  The goal of this section is to describe the essential image of the functor
	\[ \mathrm{S}_{\tilde{X}}:\Higgs^\Hsm(\X)\hookrightarrow \mathrm{Bun}(\X_v).\]
 For this we will use the comparison to 
   $\LS^{\mathrm F}_{\tilde{X}}$ and the fact that every v-vector bundle becomes Faltings-small locally on toric towers over $X$ by \Cref{l:locally-Faltings-small}.
This will allow us to describe the essential image of $\LS_{\tilde{X}}$ using the Hitchin fibration, as we now explain.

\begin{definition}\label{def:Hitchin-fibration}
Denote by $\mathrm{Bun}_n(\X_v)$ the category of v-vector bundles of rank $n$ on $\X$, and similarly let $\Higgs_n(\X)$ be the category of Higgs bundles of rank $n$.
Let $\mathrm{A}_n(\X)$ be the Hitchin base of $\X$, defined as the $\O(\X)$-module
\[ \textstyle\mathrm A_{n}(\X):=\bigoplus_{i=1}^n H^0(\X,\mathrm{Sym}^i(\Omega_\X^1(-1))).\]
Consider the set $\mathrm{A}_n(\X)$ as a groupoid where any morphism is the identity, then we can make sense of the Hitchin fibration as being a functor
\[\mathcal H:\Higgs_n(\X)\to  \mathrm{A}_n(\X).\]
From \cite{heuer-sheafified-paCS}, on the side of v-vector bundles, we have an analogous functor
\[ \wt{\mathcal H}:\mathrm{Bun}_n(\X_v)\to \mathrm{A}_n(\X)\]
defined as the composition of $\mathcal H$ with the map $\HTlog$ (see \Cref{c:comp-of-LS-prism-with-HTlog}). 

We recall from \cite{heuer-sheafified-paCS} that $\wt{\mathcal H}$ can made explicit as follows: Let $V$ be a \mbox{v-vector} bundle on $\X$. Fix a toric chart $c$, then by \Cref{l:locally-Faltings-small} we can find a toric cover $\X_n\to \X$ on which $V$ becomes Faltings-small. We can thus associate to $V$ a Higgs bundle $(\LS^{\mathrm F}_{c})^{-1}(V)$ by \Cref{t:LSF-is-equivalence}. Via \Cref{l:comp-Higgs-v-bundle-on-wtX}, this endows $V(\wt \X)$ with a Higgs field
\[ \theta_n:V(\wt \X)\to V(\wt \X) \otimes \Omega_X^1(-1).\]
Since the image of $\Omega_X^1$ in $\Omega^1_{X_n}$ is precisely  $p^{-n}\Omega^1_{X_n}$, we can rescale this by $\theta:=p^{-n}\theta_n$ to obtain a Higgs field $\theta$ that is independent of $n$. Then $\widetilde{\mathcal H}(V)$ is given by the characteristic polynomial of $\theta$. It is clear from \Cref{l:comp-Higgs-v-bundle-on-wtX} that $\theta$ in fact commutes with the $\Delta$-action, thus defining a Higgs field
$V\to V\otimes \Omega_X^1(-1)$.
and it follows that $\theta$ indeed has coefficients in $\mathrm{A}_n(\X)$. 
\end{definition}
\begin{remark}
	It is clear from the construction that the above Higgs field on $V$ is precisely $\theta_V$ from \Cref{t:canonical-HF}. This shows that we can more conceptually regard $\widetilde{\mathcal H}$ as computing the characteristic polynomial of $\theta_V$. However, the definition of $\widetilde{\mathcal H}$ that we will use later is the one given above.
\end{remark}
\begin{lemma}\label{l:comp-with-Hitchin-fibration}
	For any $n\in \N$, the following diagram commutes:
	\[\begin{tikzcd}[column sep = 0.3cm,row sep =0.3cm]
		& \mathrm{A}_n(\X)                                 &                      \\
		\Higgs^{\Hsm}_n(\X) \arrow[ru,"\mathcal H"]\arrow[rr,"\mathrm{S}_{\tilde{X}}"] &                                      & \mathrm{Bun}_n(\X_v) \arrow[lu,"\widetilde{\mathcal H}"']         
	\end{tikzcd}\]
\end{lemma}
\begin{proof}
	Immediate from 
	\Cref{c:comp-of-LS-prism-with-HTlog} and the definition of $\wt{\mathcal H}$ via $\HTlog$.
\end{proof}
\begin{definition}\label{def:Hitchin-small-locus}
	\begin{enumerate}
		\item For any $\alpha\in \Q$, let $p^{<\alpha}:=\{x\in S\text{ s.t. } \|x\|< |p|^{\alpha}\}$. We call the subspace \[\textstyle\mathrm{A}^{\Hsm}_n(\X):=\bigoplus\limits_{i=1}^n  H^0(X,p^{<\frac{i}{p-1}}\cdot\mathrm{Sym}^i(\Omega_X^1\{-1\}))\subseteq \mathrm{A}_n(\X)\] the \textit{Hitchin--small locus}. If $X$ is proper, this is an open submodule of $\mathrm{A}_n(\X)$.
		\item 
		We call a v-vector bundle $V$ on $\X$ \textit{Hitchin-small} if $\widetilde{\mathcal H}(V)$ already lies in $\mathrm{A}^{\Hsm}_n(\X)\subseteq \mathrm{A}_n(\X)$. Let $\mathrm{Bun}^{\Hsm}(\X_v)\subseteq \mathrm{Bun}(\X_v)$ be the full subcategory of Hitchin--small v-vector bundles. 
		\item We call a Higgs bundle $(E,\theta)$ on $\X$ \textit{Hitchin-small} if ${\mathcal H}(E,\theta)$ lies in $\mathrm{A}^{\Hsm}_n(\X)\subseteq \mathrm{A}_n(\X)$. Let $\mathrm{Higgs}^{\Hsm}(\X)\subseteq \mathrm{Higgs}(\X)$ be the full subcategory of Hitchin--small Higgs bundles. 
		\item More generally, for $z\in S\tf$, we say that $V$ is $z$-Hitchin small if $\widetilde{\mathcal H}(V)$ lies in 
		\[\textstyle\mathrm{A}^{\zHsm}_n(\X):=\bigoplus\limits_{i=1}^n  H^0(X,z^i\cdot p^{<1}\cdot\mathrm{Sym}^i(\Omega_X^1\{-1\}))\subseteq \mathrm{A}_n(\X),\]
		and similarly for Higgs bundles. We define $\mathrm{Bun}^{\zHsm}(\X_v)$ and  $\mathrm{Higgs}^{\zHsm}(\X)$ accordingly.
	\end{enumerate}
\end{definition}
Part (3) generalises \Cref{d:Hitchin-small} as we no longer assume that $\Z_p^\cycl\subseteq S$.

Note that since $\mathcal{H}$ localises on $\mathcal X$ by definition, being Hitchin-small is a local notion, i.e., a v-vector bundle on $\X$ is $z$-Hitchin-small if and only if it is so locally on $X$, and similarly for Higgs bundles. Moreover, we have the analogue of \Cref{l:Faltings-implies-Hitchin-small} for v-vector bundles:
\begin{lemma}\label{l:locally-Hitchin-small}
	Any Faltings-small v-vector bundle is Hitchin-small. In particular, if $X$ is toric, then any v-vector bundle on $\X$ becomes Hitchin-small after pullback to a finite toric cover $\X_n\to \X$.
\end{lemma}
\begin{proof}
	The first part follows from \Cref{p:comp-LS_F-LS_prism} and \Cref{l:comp-with-Hitchin-fibration}. The second part then follows from \Cref{l:locally-Faltings-small}.
\end{proof}

We can now prove the main result of this section:
\begin{theorem}\label{t:essential-surjectivity-geom-case}
	Let $X$ be any smoothoid formal scheme over a $p$-adic perfectoid base $S$. Let $\X$ be its adic generic fibre. Let $x\in S\tf^\times$ be such that $S\subseteq xS$. Let $a\in S$ be such that $\|a\|\leq \frac{1}{p-1}$ and set $z:=ax^{-1}$. For example, if $\zeta_p\in S$, we can simply take $a:=1-\zeta_p$.
	\begin{enumerate}
	 \item  Let $\tilde{X}$ be an $x$-lift of $X$ to $A_2(x)$ (in the sense of \cite[Definition 7.2]{AHLB}), then $\mathrm{S}_{\tilde{X},z}$ induces  a natural equivalence of categories
	 \[ \mathrm{S}_{\tilde{X},z}:  \Higgs^\zHsm(\X) \isomarrow \mathrm{Bun}^{\zHsm}(\X_v)      \]
	 \item $S_{\tilde{X},z}$ is natural in $\tilde{X}$ in the following sense: For any morphism $f:Y\to X$ of smoothoid formal schemes, any lift to an $A_2(x)$-linear morphism $\tilde{f}:\tilde{Y}\to\tilde{X}$ induces a natural transformation $\gamma(\tilde{f}):f^{\ast}\circ \mathrm S_{\tilde{X},z}\Rightarrow \mathrm S_{\tilde{Y},z}\circ f^{\ast}$ that is compatible with composition.
	 \end{enumerate}
 \end{theorem}

\begin{corollary}\label{c:induced-by-integral-model}
		Let $K$ be the completion of an algebraic closure of a $p$-adic field $K_0$.
		Assume that $X$ is a smooth $p$-adic formal scheme over $\O_K$ that  has a model $X_0$ over $\O_{K_0}$. By \cite[Corollary~7.4]{AHLB}, this induces an $e^{-1}$-lift $\tilde{X}$ of $X$ for some $e\in \O_K$. Then for $z:=(1-\zeta_p)e$, \Cref{t:essential-surjectivity-geom-case}.2 yields a canonical functor
		\[ \mathrm{S}_{\tilde{X},z}:  \Higgs^{z\text{-}\Hsm}(\X) \isomarrow \mathrm{Bun}^{z\text{-}\Hsm}(\X_v).\]
\end{corollary}

This gives a completely canonical variant of the global Simpson correspondence for $e$-Hitchin small bundles in the case of arithmetic models. For example, in the setting of \Cref{c:induced-by-integral-model}, we will see in \Cref{l:pullback-of-v-vector-bundle-to-OC-nilpotent} that any v-vector bundle on $\X$ that comes via pullback from $\X_0$ lies in the image of $\mathrm S_{\tilde{X}}$. This generalises the $p$-adic Simpson functor of Liu--Zhu \cite{LiuZhu_RiemannHilbert} (in the case of good reduction) from $\Q_p$-local systems on $X_0$ to v-vector bundles on $X_0$.

\begin{proof}
	 We already know from \cite[Thm 7.13]{AHLB} that $\mathrm{S}_{\tilde{X},z}$ is fully faithful and natural in $\tilde{X}$ as described in (2). This naturality means that $S_{\wt X,z}$ identifies descent data on both sides: For this reason, we may without loss of generality assume that $S$ contains $\Z_p^\cycl$. Here we use that for the base-change $\X_{\Q_p^\cycl}$ of $\X$ along $\Spa(\Q_p^\cycl)\to \Spa(\Q_p)$, we have descent of (analytic) Higgs bundles along the v-cover $\X_{\Q_p^\cycl}\to \X$ by \cite[Corollary~7.4]{heuer-sheafified-paCS}.
	
	It thus remains to prove that $S_{\wt X,z}$ is essentially surjective when $S$ contains $\Z_p^\cycl$. The statement is Zariski-local on $X$, so we may assume that $X$ admits a toric chart and make choices as in the beginning of \S\ref{sec:explicit-toric-charts}. In particular, we obtain over $X$ a second $A_2(x)$-lift $\tilde{X}'$ induced by the toric chart via the natural map $A_2\to A_2(x)$. But since $X$ is smooth and $A_2(x)\to R$ is a square-zero thickening, any two $A_2$-lifts over the affine space $X$ are isomorphic. By naturality of $S_{\wt X,z}$ in the lift, this means that $S_{\wt X,z}\simeq S_{\wt X',z}$. To determine the essential image, we may therefore without loss of generality fix an isomorphism $\wt X\cong \wt X'$ and assume that $\wt X$ is as described in  \S\ref{sec:explicit-toric-charts}.
	
	Let now $V$ be a Hitchin-small v-vector bundle on $\X$. Choose a toric chart and let $M_\infty:=V(\X_\infty)$. As in \Cref{def:Hitchin-fibration}, we can endow $M_\infty$ with a Higgs field $\theta:M_\infty\to M_\infty\otimes_R \Omega_X^1(-1)$ as follows: After  passing to a cover $\X_n\to \X$ to make $V$ Faltings-small using \Cref{l:locally-Faltings-small}, we can use \Cref{l:comp-Higgs-v-bundle-on-wtX} to identify $M_\infty$ with the pullback of a Faltings-small Higgs bundle from $\mathcal X_n\to \X$, hence endowing it with a Higgs field $\theta$. Then by construction, $\mathcal H(V)$ is the characteristic polynomial of $\theta$. The assumption that $V$ is $z$-Hitchin-small means now that $\frac{1}{z}\theta$ is topologically nilpotent. Since $S\subseteq xS$, we have $zS\subseteq (1-\zeta_p)S$, so this implies that $(1-\zeta_p)^{-1}\theta$ is topologically nilpotent. 
	
	Note that by \Cref{t:canonical-HF}, the map $\theta:M_\infty\to M_\infty\otimes \Omega_X^1(-1)$ descends to a Higgs field $V\to V\otimes \Omega_X^1(-1)$. In other words, this means that $\theta$ is $\Delta$-equivariant.
	
	Recall now  that the homomorphism $D=D_c:\Delta\to \Hom(\Omega^1_X,R\{1\})$ from \eqref{d:def-D}. For any $\sigma \in \Delta$, we can compose $D(\sigma)$  with $\theta$ to obtain an $R_\infty$-linear $\Delta$-equivariant endomorphism \[D(\sigma)\circ \theta:M_\infty\to M_\infty.\] 
	The fact that $(1-\zeta_p)^{-1}\theta$ is topologically nilpotent now ensures that the map
	\[ \exp(D\circ \theta):\Delta\to \GL(M_\infty)\]
	converges. We claim that it is indeed a continuous 1-cocycle for the $\Delta$-action on $\GL(M_\infty)$ via conjugation. Indeed, for any $\sigma_1,\sigma_2\in \Delta$, due to the fact that $D(\sigma_2)\circ \theta$ commutes with $\sigma_1$ on $M_\infty$,
	\begin{align*}
	\exp(D(\sigma_1)\circ \theta)\cdot \sigma_1\exp(D(\sigma_2)\circ \theta)\sigma_1^{-1}&= \exp(D(\sigma_1)\circ \theta)\cdot \exp( \sigma_1 (D(\sigma_2)\circ \theta) \sigma_1^{-1})\\&=\exp(D(\sigma_1)\circ \theta)\cdot \exp(D(\sigma_2)\circ \theta)\\&=\exp(D(\sigma_1+\sigma_2)\circ \theta).
	\end{align*}
	
	This enables us to reverse the construction of a v-vector bundle out of a Higgs bundle described in \Cref{t:LSF-is-equivalence}, as follows. To find a preimage of $V$ under $\mathrm{S}_{\tilde{X},z}$, we need to find an $R\tf$-submodule $M\subseteq M_\infty$ of the same rank $m$ as $M_\infty$ over $R_\infty\tf$ such that for any $\sigma\in \Delta$, the following diagram commutes:
	\[
	\begin{tikzcd}[column sep = 1.5cm]
		M\otimes_RR_\infty \arrow[d] \arrow[r, "\id\otimes\sigma"] & M\otimes_RR_\infty \arrow[r, "\exp(D(\sigma)\circ \theta)"] & M\otimes_RR_\infty \arrow[d] \\
		M_\infty \arrow[rr, "\sigma"]                              &                                                             & M_\infty.                    
	\end{tikzcd}\]
	This can be reformulated as saying that $M$ is the submodule fixed by the $R_\infty\tf$-semilinear $\Delta$-action on $V(\X_\infty)$ defined by
	\[ \sigma\ast v=\exp(D(\sigma)\circ \theta)^{-1}\sigma v.\]
	This indeed defines a $\Delta$-action because $\exp(D\circ \theta)$ is a $1$-cocycle. Thus the $R\tf$-module $M$ is uniquely determined by the property that with respect to this twisted action, it satisfies
	\[ M=M_\infty^\Delta.\]
	By the Higgs field condition and the fact that the $\Delta$-action commutes with $\theta$, also the $\ast$-action commutes with $\theta$, so by taking $\Delta$-invariants of $\theta$,  we obtain a Higgs field $\theta_M:M\to M\otimes \Omega_X^1$.
	
	It remains to prove that the rank of $M$ is $m$. To see this, consider the action of the subgroup $\Delta_n$, corresponding to the pullback of $V$ to $\X_n$. By \Cref{l:locally-Faltings-small}, the v-vector bundle becomes Faltings-small on $\X_n\to \X$ for some $n$. It follows from the fact that $\LS_{c}^{\mathrm F}$ is an equivalence that there exists an $R_n\tf$-submodule $M_n\subset M_\infty$ of rank $m$ such that $M_n=M_\infty^{\Delta_n}$  by \Cref{p:comp-LS_F-LS_prism} and  \Cref{l:comp-Higgs-v-bundle-on-wtX}. We thus have
	\[M_\infty^\Delta=(M_n)^{\Delta/\Delta_n}.\]
	But the $R_n\tf$-semilinear action of the finite group $\Delta/\Delta_n$ now defines an \'etale descent datum for a vector bundle of rank $m$ along the $\Delta/\Delta_n$-Galois cover $\X_n\to \X$. By \cite[Lemma~2.6.5]{heuer-sheafified-paCS} (which is based on \cite[Theorem 8.2.22.(d)]{kedlaya_liu_relative_p_adic_hodge_theory_foundations}), this descent datum is effective and thus gives the desired $R\tf$-module $M$ of rank $m$.
	
	At this point, we have found a Higgs bundle $(M,\theta_M)$ such that $S_{\wt X,1}(M,\theta_M)=V$. It remains to observe that the nilpotence of $\frac{1}{z}\theta$ implies that $\frac{1}{z}\theta_M$ is nilpotent, hence $(M,\theta_M)$ is $z$-Hitchin small. It follows that $S_{\wt X,z}(M,\theta_M)=S_{\wt X,1}(M,\theta_M)=V$ is in the essential image of $S_{\wt X,z}$.
\end{proof}
\begin{remark}
	In fact, one can show that $D:\Delta\to \Hom(\Omega^1_X,R_\infty\{1\})$ is divisible by $(1-\zeta_p)$. Namely, in the setup of  (cf \cite[\S4.2]{AHLB}), $D(\sigma)$ can be described as sending $\frac{dT_i}{T_i}$ to $\mu^{c_i(\sigma)}[T_i^\flat]$, whose image in $R_\infty\{1\}$  is divisible by $(1-\zeta_p)$.
	This means that the above proof actually only used that $(1-\zeta_p)^{-1}\theta$ is nilpotent in the last paragraph, as for the convergence of $\exp(D\circ \theta)$ it would suffice that only  $\theta$ is nilpotent. Indeed, the above argument can therefore be used to show that the image of the \textit{local} correspondence \cite[Corollary~6.30]{AHLB} is given by v-vector bundles for which $\theta$ is nilpotent. The point is that the  stronger condition on $\theta$ of Hitchin-smallness is necessary to define the global  $p$-adic Simpson functor $S_{\tilde{X}}$ in the first place.
\end{remark}
\begin{remark}
	In \cite{AHLB}, we more generally worked in the derived category and constructed a functor 
	\[ \mathrm{S}_{\tilde{X}}:\{\text{Hitchin-small Higgs perfect complexes}\}\to \mathcal Perf(\X_v).\]
	We believe that it should be possible in principle to describe the essential image of this functor in a similar way. For this one would first have to show that every perfect complex on $\mathcal Perf(\X_v)$ carries a natural Higgs field.
	However, already describing the essential image of the local $p$-adic Simpson functor for coherent Higgs modules is likely more difficult, cf.\ \cite[Corollary 6.30]{AHLB}. 
\end{remark}

As an application of the generality of smoothoid spaces that we work with in this section, we can deduce a more geometric version in terms of moduli stacks: For this we consider for any $n\in \N$ the functors fibred in groupoids on the category of perfectoid spaces over $K$
\begin{eqnarray*}
 \sBun_{n}:Y&\mapsto& \{ \text{v-vector bundles on $\X\times Y$ of rank $n$}\},\\
 \sHiggs_{n}:Y&\mapsto& \{ \text{Higgs bundles on $\X\times Y$ of rank $n$}\}.
\end{eqnarray*}
By \cite[Thm~1.4]{heuer-sheafified-paCS}, these are small v-stacks, and the Hitchin fibrations from the last subsection can be assembled into morphisms of v-stacks over the Hitchin base:
\begin{definition}
	The \textit{Hitchin base} of $\X$ is the v-stack $\mathcal A_{n}$ defined by sending any perfectoid space $Y$ over $K$ to $\mathrm{A}_n(X\times Y)$ where $\mathrm{A}_n$ was defined in \Cref{def:Hitchin-small-locus}. Let \[\mathcal A_{n}^{\zHsm}\subseteq \mathcal A_{n}\]
	 be the Hitchin small locus given on $Y$ by $A_{n}^{\zHsm}(X\times Y)$  defined as in \Cref{def:Hitchin-small-locus}.
\end{definition}

The Hitchin fibrations are now morphisms of v-stacks over the same base
\[
	\begin{tikzcd}[row sep =-0.1cm,column sep = 1cm]
		{\sBun_{n}} \arrow[rd,"\widetilde{\mathcal H}"] &   \\                        &  \mathcal A_{n}\\
		\sHiggs_{n}\arrow[ru,"\mathcal H"']  &          
	\end{tikzcd}
\]
We thus obtain moduli spaces of Hitchin-small objects
\[ \sHiggs_{n}^\zHsm:=\sHiggs_{n}\times_{\mathcal{A}_n}\mathcal{A}_n^{\zHsm},\quad  \sBun^\zHsm_{n}:=\sBun_{n}\times_{\mathcal{A}_n}\mathcal{A}_n^{\zHsm}\]
as the respective preimages of the Hitchin-small locus.
\begin{corollary}
	\label{cor:smoothoid-case-2-comparison-of-moduli-spaces-over-the-hitchin-small-locus}
	Let $X$ be a $p$-adic smooth formal scheme over $\O_K$ for some perfectoid field $K$ containing $\Q_p^\cycl$. Then any choice of lift $\tilde{X}$ of $X$ to $A_2(x)$ induces an equivalence of small v-stacks
	\[\mathcal{S}_{\tilde{X},z}:
	\sHiggs_{n}^\zHsm\isomarrow \sBun^\zHsm_{n}
	\]
\end{corollary}
\begin{proof}
	For any affine perfectoid formal scheme $Y=\Spf(S^+)$ over $\O_K$, the lift $\tilde{X}$ induces a flat lift of $X\times_{\O_K}Y$ to $A_2$ given by $\tilde{X}\times_{\Spf(A)}\Spf(A_{\inf}(S^+))$, and this formation is functorial in $Y$. The statement thus follows from applying \Cref{t:essential-surjectivity-geom-case} to $X\times Y$ for the choice of lift just described, using the functoriality in \Cref{t:essential-surjectivity-geom-case}.(3).
\end{proof}
As a special case, we can deduce the first non-trivial example of a comparison between the full moduli stacks beyond $n=1$, namely for the special case of projective space:
\begin{corollary}
	Let $X=\mathbb P^m_{\O_K}$, then there is a canonical isomorphism of small $v$-stacks
	\[\mathcal{S}:
	\sHiggs_{n}\isomarrow \sBun_{n}.
	\]
	More generally, such an isomorphism exists for any smooth proper $p$-adic formal scheme $X$ over $\O_K$ with a lift $\tilde{X}$ of $X$ to $A_2$ such that $X$ has trivial Hitchin base $\mathcal A_{n}(K)=\{0\}$.
\end{corollary}
\begin{proof}
	For $X=\mathbb P^m$, we see from the Euler sequence that $\mathcal A_{n}(K)=\{0\}$: 
	Indeed, we have $\Omega_X^1\subseteq \O_{\mathbb P^m}(-1)^{m+1}$ and hence $\mathrm{Sym}^k\Omega_X^1\subseteq \mathrm{Sym}^k\O_{\mathbb P^m}(-1)^{m+1}\cong \O^{{m+1}\choose k}\otimes \O_{\mathbb P^m}(-1)^{\otimes k}$ has no non-trivial global sections.
	The statement now follows from \Cref{cor:smoothoid-case-2-comparison-of-moduli-spaces-over-the-hitchin-small-locus}: For any affine perfectoid $Y$, any Higgs bundle, respectively v-vector bundle, on $X\times Y$ is in the fibre of $0$, hence is Hitchin-small.
\end{proof} 
  
As a second application, we deduce the following analogue of \cite[Thm~1.2]{analytic_HT}:
\begin{corollary}
	Let $X$ be a toric smoothoid formal scheme over a perfectoid base $S$ that contains $\Z_p^\cycl$. Choose a toric chart $c:X\to \mathbb T^d_S$ and consider the induced toric tower $\dots \to X_{n+1}\to X_n\to \dots \to X$ of formal schemes, so that the adic generic fibre of $X_n\to X$ is Galois with group $\Delta_n$. Then there is an equivalence of categories
\[ 2\text{-}\varinjlim\limits_{n\in\N} \Vec([X_n^{\rm HT}/\Delta_n],\O\tf)\isomarrow \mathrm{Bun}(\X_v) \]
\end{corollary}
\begin{proof}
	By Galois descent, we have a canonical functor from left to right, which is fully faithful by \cite[Thm 7.13]{AHLB}.
	 By \Cref{t:essential-surjectivity-geom-case}, any Hitchin-small v-vector bundle is in the essential image. By \Cref{l:locally-Hitchin-small}, any v-vector bundle becomes Hitchin-small on some $\X_n$.
\end{proof}

\section{Twisted pullback}\label{s:twisted-pullback}
As the second goal for the geometric \Cref{setup:geometric}, we now clarify the functoriality of our constructions, i.e.\ how the small $p$-adic Simpson correspondences we defined  can be compared under morphisms of smoothoids $f:Y\to X$. While \Cref{t:essential-surjectivity-geom-case}.(2)  explains how to do this if there exists an $A_2$-lift of $f$, it is interesting for applications what one can say in the absence of such an $A_2$-lift of $f$. To give a concrete example, in arithmetic situations, one can often obtain lifts to $B^+_{\dR}/\xi^2$ by base-change from arithmetic base fields, but these may not respect the integral structures. A solution to this issue is then to choose integral lifts of $X$ and $Y$, but this may not be possible in such a way that a lift of $f$ exists.

Indeed, for this reason, the question of functoriality of the $p$-adic Simpson correspondence is already treated by Faltings (in his slightly different setting of what we call Faltings-small bundles): It is very briefly sketched in three sentences on \cite[p855]{faltings2005p} and is implicit in construction of ``twisted pullback'' \cite[p858]{faltings2005p}. In \cite{faltings2005p}, this is further developed by way of explicit formulas in the proof of the full $p$-adic Simpson correspondence for curves. The goal of this section can also be described as reinterpreting Faltings' definition in a more geometric way.

\medskip

Let $S$ be a $p$-adic perfectoid base over $\Z_p$ and let $f:Y\to X$ be a morphism of smoothoid formal schemes over $S$. Let $\X$ and $\mathcal Y$ be the respective generic fibres. Given any $A_2$-lifts $\tilde X$ of $X$ and $\tilde Y$ of $Y$ respectively, the functors $\beta_{\tilde X}\tf$ and $\beta_{\tilde Y}\tf$ from \eqref{eq:def-alpha-beta} fit by \cite[Proposition 6.28]{AHLB} and functoriality of the Hodge--Tate stack into a commutative diagram
\[
\begin{tikzcd}
	\Higgs^\Hsm(\mathcal Y) & \Bun(z_{\ast}X^\HT)\tf. \arrow[l,"\beta_{\tilde Y}"',"\sim"]\\
	\Higgs^\Hsm(\X)  \arrow[dotted,u]& \Bun(z_{\ast}X^\HT)\tf \arrow[u,"f^{\ast}"] \arrow[l,"\beta_{\tilde X}"',"\sim"]
\end{tikzcd}\]
\begin{definition}\label{d:twisted-pullback}
	We denote by $f^{\circ}_{\tilde X,\tilde Y}:	\Higgs^\Hsm(\X)\to \Higgs^\Hsm(\mathcal Y)$ the composition $f^{\circ}_{\tilde X,\tilde Y}:=\beta_{\tilde Y}\circ f^\ast \circ \beta_{\tilde X}^{-1}$, which is the essentially unique dotted arrow making the above diagram commute.
\end{definition}
The goal of this section is to give a more explicit description of this functor. For this, we use a construction due to Abbes--Gros  \cite[II.10.3]{abbes2016p}:
\begin{definition}
	The Higgs--Tate torsor $\mathcal {HT}_f$ is the sheaf on $Y_{\mathrm{Zar}}$ defined as the subsheaf of 
	\[ \underline{\Hom}(f^{-1}\O_{\tilde X},\O_{\tilde Y})\]
	of $A_2$-algebra homomorphisms that reduce to the map $f^{-1}\O_{X}\to \O_{Y}$ defined by $f:Y\to X$.
\end{definition}
\begin{lemma}
	$\mathcal {HT}_f$ is a  $f^{\ast}T_X\{1\}$-torsor on $Y$.
\end{lemma}
\begin{proof}
	The statement is local on $Y$, so we may assume that $X$ and $Y$ are affine. Then a global section of $\mathcal {HT}_f(Y)$ exists by formal smoothness. By standard deformation theory, the space of $A_2$-lifts is then a principal homogeneous space under the group of derivations 
	\[\O(X)\to \ker(\mathcal O(\tilde Y)\to \O(Y))=\O(Y)\{1\},\]
	which are given by $\Hom_{\O(X)}(\Omega_X(X),\O(Y)\{1\})=f^{\ast}T_X\{1\}(Y)$.
\end{proof}
\begin{proposition}\label{p:twisted-pullback}
	Let $(M,\theta)\in \Higgs^\Hsm(\X)$. Then there is a canonical isomorphism
	\[ f^{\circ}_{\tilde X,\tilde Y}(M,\theta)=f^\ast(M,\theta)\times^{T_X\{1\}}\mathcal{HT}_f\]
	where $T_X\{1\}$ acts on $(M,\theta)$ via the homomorphism \[\exp(\theta):T_X\{1\}\to \underline{\mathrm{Aut}}(M).\]
\end{proposition}
For small Higgs bundles on curves of good reduction, this recovers the definition of Faltings \cite[p858]{faltings2005p}, phrased by him in terms of the spectral curve.

\begin{corollary}
	Any datum of a lift $\tilde f:\tilde X\to \tilde Y$ induces a natural equivalence $\gamma_{\tilde f}:f^{\circ}_{\tilde X,\tilde Y}\Rightarrow f^{\ast}$.
\end{corollary}
\begin{proof}[Proof of \Cref{p:twisted-pullback}]
	The statement is Zariski-local on $Y$, so we may assume that $Y=\Spf(B)\to X=\Spf(A)$ is affine. We now first prove that any lift $\tilde f:\tilde Y\to \tilde X$ induces a 2-arrow
	\[
	\begin{tikzcd}
		X \arrow[r,"\rho_{\tilde X}"]           & \Lft_{X}                                  \\
		Y \arrow[r,"\rho_{\tilde Y}"] \arrow[u] & \Lft_{Y}. \arrow[u] \arrow[lu, Rightarrow,"\gamma_{\tilde f}"',shorten >=2.5ex,shorten <=2.0ex]
	\end{tikzcd} \]
	Indeed, let $T$ be any $S$-algebra, then for any $T$-point $B\to T$ of $Y$, the image in $\Lft_{X}(T)$ defined by going around the upper left corner is  the composition
	\[ x:A\xrightarrow{f} B\xrightarrow{\delta_{\tilde Y}} B\oplus B\{1\}[1]\to T\oplus T\{1\}[1],\]
	where $\delta_{\tilde{Y}}$ is the morphism of animated rings associated to the square zero extension of $B$ defined by $\tilde{Y}$, cf.\ \cite[\S7.2]{AHLB}\cite[\S 5.1.9]{cesnavicius2019purity}.
	Going around the bottom right corner, we instead obtain the composition
	\[A\xrightarrow{\delta_{\tilde X}}A\oplus A\{1\}[1]\xrightarrow{f} B\oplus B\{1\}[1] \to T\oplus T\{1\}[1],\]
	where $\delta_{\tilde X}$ is the morphism of animated algebra defined by $\tilde{X}$.
	The lift $\tilde f$ now defines a map between
	$A\xrightarrow{f\circ \delta_{\tilde X}} B\oplus B\{1\}[1]$ and $A\xrightarrow{\delta_{\tilde Y}\circ f} B\oplus B\{1\}[1]$, which induces the desired isomorphism in $\Lft_{X}(T)$.
	
	It is clear from these explicit descriptions that for any two lifts $\tilde f_1,\tilde f_2:\tilde Y\to \tilde X$, the map 
	\begin{equation}\label{eq:homotopy-two-different-lifts}
	\gamma_{\tilde f_1}\circ \gamma_{\tilde f_2}^{-1}\in \Aut(x:Y\to X\to \Lft_{X})=T_X\{1\}(Y)
	\end{equation}
	can be identified with the difference $\tilde f_1-\tilde f_2\in T_X\{1\}(Y)$ defined by the fact that the Higgs--Tate torsor $\pi:\mathcal{HT}_f\to Y$ is a  $T_X\{1\}$-torsor.
	
	We now pass to the universal situation: There is a canonical transformation
	\[
	\begin{tikzcd}
		X \arrow[r,"\rho_{\tilde X}"]           & \Lft_{X}                                  \\
		\mathcal{HT}_f \arrow[r,"\rho_{\tilde Y}\circ \pi"'] \arrow[u,"f\circ \pi"] & \Lft_{Y} \arrow[u] \arrow[lu, Rightarrow,"\gamma"',shorten >=2.5ex,shorten <=2.0ex]
	\end{tikzcd} \]
	such that for any lift $\tilde f\in  \mathcal{HT}_f(Y)$, we obtain $\gamma_{\tilde f}$ by specialisation at the section $\tilde f:Y\to \mathcal{HT}_f$. This is clearly natural in $X$ and $Y$, so it glues beyond the affine case, and we may now assume that $X$ and $Y$ are general. The commutativity of the diagram then says that for any small Higgs bundle $(M,\theta)$ on $X$, there is a canonical isomorphism over $\mathcal{HT}_f$
	\[ \varphi_\gamma:\pi^{\ast}f^{\circ}M\isomarrow\pi^{\ast}f^{\ast}M .\]
	We now compute the effect which the $T_X\{1\}$-action on $\mathcal{HT}_f$ has on this isomorphism: Let $\delta \in T_X\{1\}(Y)$ be any local section and consider the diagram
	\[
	\begin{tikzcd}
		X \arrow[r,equal]           &X \arrow[r]           & \Lft_{X}                                  \\
		\mathcal{HT}_f \arrow[r,"\delta\cdot "] \arrow[u,"f\circ \pi"] &\mathcal{HT}_f \arrow[r] \arrow[u,"f\circ \pi"] & \Lft_{Y}. \arrow[u] \arrow[lu, Rightarrow,"\gamma"',shorten >=2.5ex,shorten <=2.0ex]
	\end{tikzcd} \]
	Let $\gamma_{\delta}$ be the resulting homotopy between the two outer compositions. We wish to compute 
	\[\gamma_{\delta}\circ \gamma^{-1}\in  \Aut_Y(\mathcal{HT}_f\to \Lft_{X})=T_X\{1\}(\mathcal{HT}_f).\]
	 For this we use that locally on $Y$ where $\mathcal{HT}_f$ admits a section $\wt f\in \mathcal{HT}_f(Y)$, the above diagram becomes isomorphic to 
	\[
	\begin{tikzcd}
			X \arrow[r,equal]           &X \arrow[r,equal]           &X \arrow[r]           & \Lft_{X}                                  \\
			f^\ast T_X\{1\} \arrow[r,"\delta\cdot "] \arrow[u,"f\circ \pi"] &f^\ast T_X\{1\} \arrow[r] \arrow[u,"f\circ \pi"] &Y \arrow[r] \arrow[u,"f"] & \Lft_{Y}. \arrow[u] \arrow[lu, Rightarrow,"\gamma_{\wt f}"',shorten >=2.5ex,shorten <=2.0ex]
	\end{tikzcd} \]
	If follows that $\gamma_{\delta}\circ \gamma^{-1}$ can be decomposed into the difference $\gamma_{\wt f+\delta}\circ \gamma^{-1}_{\wt f}$ composed with $f^\ast T_X\{1\}$, plus the difference between $\mathcal{HT}_f\xrightarrow{\delta}\mathcal{HT}_f\to Y$ and  $\mathcal{HT}_f\to Y$.
	By the computation surrounding \eqref{eq:homotopy-two-different-lifts}, we have  $\gamma_{\wt f+\delta}\circ \gamma^{-1}_{\wt f}=\delta$. It follows that $\gamma_{\delta}\circ \gamma^{-1}$ is given by the image of $(\delta,\delta)$ under
	\[\Aut(Y\to \Lft_{X})\times \Aut_Y(\mathcal{HT}_f) = T_X\{1\}(Y)\times T_X\{1\}(Y) \xrightarrow{m} T_X\{1\}(\mathcal{HT}_f).\]
	
	As a consequence, we can regard $\varphi_\gamma$ as a $T_X\{1\}$-equivariant isomorphism 
	\[\varphi_\gamma: \pi^{-1}f^{\circ}M\otimes_{\pi^{-1}\O_Y}\O_{\mathcal{HT}_f}\isomarrow \pi^{-1}f^{\ast}M\otimes_{\pi^{-1}\O_Y}\O_{\mathcal{HT}_f}\]
	where the left hand side is endowed with the natural action via the second factor, whereas the right hand side is endowed with the diagonal action of $T_X$. Here the action on the first factor $\pi^{-1}f^{\ast}M$ has to be the one induced by $T_X\{1\}=\Aut(x:Y\to \Lft_{X})$. By \Cref{l:action-via-exp} below, this is given by the natural map $\exp:T_X\to \Aut(M)$. 
	
	Taking $T_X\{1\}$-equivariant sections on both sides with respect to the natural actions just described, this defines a natural isomorphism
	\[ f^{\circ}M=(\pi^{\ast}f^{\circ}M)^{T_X\{1\}}\xisomarrow{\varphi_\gamma}( \pi^{-1}f^{\ast}M\otimes_{\pi^{-1}\O_Y}\O_{\mathcal{HT}_f})^{T_X\{1\}}=f^{\ast}M\times ^{T_X\{1\}}\mathcal{HT}_f\qedhere\]
\end{proof}
\begin{lemma}\label{l:action-via-exp}
	Assume $X=\Spf(R)$ is affine. Let $N\in \Bun(BT_X^\sharp\{1\})$ and let $(M,\theta)=\Phi_X(N)$ be the corresponding Higgs bundle via \eqref{eq:constr-of-Higgs-field-via-Cariter-duality}. Then for any $p$-complete $R$-algebra $S$ and any $s\in T_X^\sharp(S)$, the  action of $s$ on $M$ obtained from the $T_X^\sharp\{1\}$-action on $X\to BT_X^\sharp\{1\}$ is given by
	\[\exp(s\circ \theta):M\otimes_R S\isomarrow M\otimes_RS.\]
	Here we have identified $s$ first with the induced $\O_X$-linear map $s:\wtOm_X\to \O_X$ and then the map $ s\circ \theta$ is the resulting composition $(\id \otimes s)\circ \theta:M\to M\otimes \wtOm\to M \in \End(M)$.
\end{lemma}
\begin{proof}
	The $T_X^\sharp\{1\}$-action on $M$ is encoded on the level of sheaves by a co-action map $M\to M\otimes \wtOm_X$.
	Unravelling the definition of $\Phi_X$ in \eqref{eq:constr-of-Higgs-field-via-Cariter-duality}, we see that the action of $\Symvw(\wtOm_X^\vee)$ on $M$ is obtained from this via duality: More precisely, like in \eqref{eq:expl-descr-of-phi}, we see from \cite[(9) in the proof of Lemma~6.4]{AHLB} that this identifies $s:\wtOm_X\to S$ with the map $\exp(\langle - ,s\rangle):\widehat{\mathrm{Sym}^\bullet_X}(\wtOm^\vee)\to \mathbb G_m$, which corresponds to the section $\exp(s)\in \Symvw(\wtOm^\vee)\otimes S$.
\end{proof}

We can now use twisted pullback to explain the functoriality of the small global $p$-adic Simpson correspondence, generalizing the statement of \Cref{t:essential-surjectivity-geom-case}.(2): 

\begin{proposition}
	Let $f:Y\to X$ be any morphism of smoothoid $p$-adic formal schemes. Let $\tilde X$ be an $A_2$-lift of $X$ and let $\tilde Y$ be an $A_2$-lift of $Y$. Let $\mathcal X$ and $\mathcal Y$ be the respective adic generic fibres of $X$ and $Y$. Then there is a canonical natural transformation
	\[
	\begin{tikzcd}
				{\Higgs^\Hsm(\mathcal Y)} \arrow[r,"\mathrm{S}_{\tilde{Y}}"]          & {\mathrm{Bun}^{\Hsm}(\mathcal Y_v)}   \\ \Higgs^\Hsm(\X)\arrow[r,"\mathrm{S}_{\tilde{X}}"]
		\arrow[u,"f^\circ_{\tilde X,\tilde Y}"]  & {\mathrm{Bun}^{\Hsm}(\X_v)} \arrow[u,"f^\ast"]\arrow[lu,Rightarrow,shorten >=3.5ex,shorten <=3.0ex]  
	\end{tikzcd}\]
\end{proposition}
\begin{proof}
	By definition of $\mathrm{S}_{\tilde{X}}$,
	this follows by considering the natural diagram:
	\[
	\begin{tikzcd}		\Higgs^\Hsm(\mathcal Y)  & \Bun(z_{\ast}Y^\HT)\tf \arrow[r,"\alpha^{\ast}"]\arrow[l,"\beta_{\tilde Y}"',"\sim"] &\mathrm{Bun}^{\Hsm}(\mathcal Y_v)\\
		\Higgs^\Hsm(\X)  \arrow[u,"f^\circ_{\tilde X,\tilde Y}"]& \Bun(z_{\ast}X^\HT)\tf \arrow[l,"\beta_{\tilde X}"',"\sim"] \arrow[r,"\alpha^{\ast}"]\arrow[u,"f^{\ast}"] &\mathrm{Bun}^{\Hsm}(\mathcal X_v) \arrow[u,"f^{\ast}"]
	\end{tikzcd}\]
	The left diagram commutes by \Cref{d:twisted-pullback}. The right diagram commutes up to a canonical natural equivalence by naturality of $\alpha$.
\end{proof}

\section{Essential image: arithmetic case}\label{sec:essential-image-arithmetic}
We now switch to the arithmetic \Cref{setup:arithmetic}. In particular, $X\to \Spf(\O_K)$ is a smooth \mbox{$p$-adic} formal scheme over $\O_K$ for a $p$-adic field $K$, and throughout we have fixed the choice of a uniformizer $\pi \in \mathcal{O}_K$. Let $k=\O_K/\pi$ be the residue field. Let $E$ be the characteristic polynomial of $\pi$ over $W(k)$ and $e=E'(\pi)$ the induced generator of the different $\delta_{\O_K|W(k)}$. Let $\X$ be the adic generic fibre of $X$.
In this section, we study the functor 
\[\mathrm{S}_\pi: \left\{ \begin{array}{@{}c@{}l}\text{Hitchin-small}\\\text{Higgs--Sen bundles on $\X$}\end{array}\right\}\to \left\{ \begin{array}{@{}c@{}l}\text{vector bundles}\\\text{on $\X_v$}\end{array}\right\}\]
whose construction via the Hodge--Tate stack we recalled in \S\ref{s:recall-global-Simpson-arith}.
In analogy to the geometric case of \S\ref{sec:essential-image-vVB-geom-case}, the goal of this section is to describe the essential image of this functor, by introducing an appropriate Hitchin fibration. 
The basic idea is to use the geometric $p$-adic Simpson correspondence on the base-change $\X_{C}$ to the completed algebraic closure $C$ of $K$, and then to argue by Galois descent. The constructions in this section are therefore closely related to those of \cite[\S7]{MinWang22}, with the Hitchin fibration being the crucial additional ingredient.

\subsection{The arithmetic Hitchin fibration for Higgs--Sen-bundles}
We begin by defining the Hitchin fibration for Higgs--Sen bundles:
\begin{definition}
	Let $n\in \N$.
	The Hitchin fibration of rank $n$ is the natural morphism of groupoids
	\[\{\text{Higgs--Sen modules on $\X$ of rank $n$}\}\to \mathbb A^n(\X)\]
	defined by sending $(N,\theta_N,\phi_N)$ to the characteristic polynomial of $\phi_N$. 
	Here and in the following, we regard $\mathbb A^n$ as the parameter space of monic polynomials  $F(T)=T^n+a_{1}T^{n-1}+\dots +a_n$ of degree $n$ in terms of the tuple $(a_1,\dots,a_n)$.
\end{definition}
\begin{definition}
	The Hitchin-small locus $\mathbb{A}^{\mathrm{sm},n}(\X)\subseteq \mathbb{A}^n(\X)$ is the open subgroup defined as
	\[\mathbb{A}^{\mathrm{sm},n}(\X):=e\Z+\mathfrak m_K\O(X)^n.\] More explicitly, if $K$ is absolutely unramified, then $e=1$ and we have $\mathbb{A}^{\mathrm{sm},n}(\X)=\Z+\mathfrak m_K\O(X)^n$. Otherwise, $e\in \mathfrak m_K$ and we thus have $\mathbb{A}^{\mathrm{sm},n}(\X)=\mathfrak m_K\O(X)^n$.
\end{definition}
\begin{definition}
	A Higgs-Sen module $(N,\theta_N,\phi_N)$ on $\X$ is Hitchin-small if $\mathcal H(N,\theta_N,\phi_N)\in \mathbb{A}^{\mathrm{sm},n}(\X)$.
\end{definition}
It will follow from \Cref{l:comp-between-Kummer-and-cycl} below that this is equivalent to the definition of Hitchin-smallness given in the introduction in terms of Hodge--Tate--Sen weights.
\begin{lemma}\label{l:small-locus-arithmetic-Hitchin}
	For an $\O(\X)$-linear map $F:\O(\X)^n\to \O(\X)^n$, the following are equivalent:
	\begin{enumerate}
		\item  $F^p-e^{p-1}F$ is topologically nilpotent.
		\item  The characteristic polynomial of $F$ lies in $\mathbb{A}^{\mathrm{sm},n}(\X)\subseteq \mathbb{A}^n(\X)$.
	\end{enumerate}
\end{lemma}
\begin{proof}
	Both properties can be checked on fibers of points $\Spa(L)\to \X$ in unramified extensions $L$ of $K$. Using that $T^p-e^{p-1}T\equiv \prod_{i=0}^{p-1}(T-e\cdot i)\bmod \pi$ in $\O_K[T]$, one verifies that  $F^p-e^{p-1}F$ is topologically nilpotent $\Leftrightarrow$ all eigenvalues of $F$ lie in 
	$e\Z+\mathfrak{m}_{L}$.
\end{proof}

\subsection{Comparison to geometric setup}
Our next goal is to define the Hitchin fibration for v-vector bundles  $\wt {\mathcal H}$  in the arithmetic setting. This is supposed to associate to any v-vector bundle $V$ of rank $n$ on $\X_v$ a spectral datum in $\O(\X)^n$, functorially in $X$, and which for $V=\mathrm{S}_\pi(N,\theta_N,\phi_N)$ is given by the characteristic polynomial of $\phi_N$.

In order to construct this, we may work locally and assume that $X$ is affinoid and admits a prismatic lift $(A,I)$, for example induced by a toric chart. We consider the base-change 
\[ h:\X_{C}\to \X\]
to the completed algebraic closure $C|K$
like in \S\ref{s:recall-global-Simpson-arith} and study the pullback $h^{\ast}V$ with its natural $G_K$-equivariant structure on $\X_C$. Recall from \S\ref{s:recall-global-Simpson-arith} that there is a canonical lift
\[ s:\O_K\to A_2(e^{-1}).\]
Via base-change, we obtain an $e$-lift
\[
\tilde X_{\mathcal{O}_C} := X \times_{\Spf(\O_K)}\Spf(A_{2}(e^{-1}))
\]
which is a Galois-equivariant lift of $X_{\mathcal{O}_C}$ to the square-zero thickening \[A_{2}(e^{-1})=e^{-1}\xi+A_{\inf}(\O_C)/\xi^2\subseteq B_{\dR}^+/\xi^2\] of $\O_C$. Let $z:=e(1-\zeta_p)$, then by \Cref{c:induced-by-integral-model}, this defines a fully faithful functor
\[\mathrm{S}_{\tilde X_{\mathcal{O}_C}}: \left\{ \begin{array}{@{}c@{}l}\text{$z$-Hitchin-small}\\\text{Higgs bundles on $\X_C$}\end{array}\right\}\to \left\{ \begin{array}{@{}c@{}l}\text{vector bundles}\\\text{on $\X_{C,v}$}\end{array}\right\}.\]
As a  first step, we now use the Galois action on $\tilde X_{\mathcal{O}_C}$ to see:
\begin{lemma}\label{l:pullback-of-v-vector-bundle-to-OC-nilpotent}
	Let $V$ be any v-vector bundle on $\X$. Then the v-vector bundle $h^{\ast}V$ on $\X_C$ is \mbox{$e$-Hitchin-small}, in fact even $0$-Hitchin small. Its associated Higgs bundle $(M,\theta_M)$ under the functor $\mathrm{S}_{\tilde X_{\mathcal{O}_C}}$ from \Cref{c:induced-by-integral-model} is nilpotent.
\end{lemma}
\begin{proof}
	This follows from functoriality of the Hitchin fibration \[\wt{\mathcal H}:\mathrm{Bun}_n(\X_{C,v})\to \mathcal A_n(\X).\]
	Indeed, by \cite[Proposition 8.13]{heuer-sheafified-paCS}, $\wt{\mathcal H}$ is equivariant for the $G_K$-action.
	 Since $h^{\ast}V$ admits a $G_K$-equivariant structure, and $\wt{\mathcal H}$ is a morphism of groupoids, it follows that $\wt{\mathcal H}(h^{\ast}V)$ is Galois invariant. But due to the presence of Tate twists in the definition of  $\mathcal A_n$, we have by \cite[\S3.3 Thm. 2]{tate1967p} that
	\[  \mathcal A_n(C)^{G_K}=\oplus_{i=1}^n \big( H^0(\X,\mathrm{Sym}^i(\Omega_\X^1))\otimes_K C(-i)\big)^{G_K} =0.\]
	In particular, $\wt{\mathcal H}(h^{\ast}V)\in \mathcal A_n^\Hsm(C)$, hence $h^\ast V$ is Hitchin small.
	By \Cref{l:comp-with-Hitchin-fibration}, it also follows that the associated Higgs bundle $(M,\theta_M)$ has characteristic polynomial $T^n$, hence $\theta_M^n=0$.
\end{proof}

By the naturality of the functor $\mathrm{S}_{\tilde{X}}$ with respect to the datum of $\tilde{X}$ explained in \Cref{t:essential-surjectivity-geom-case}.2, it follows that the natural Galois action on $\tilde X_{\mathcal{O}_C}$ and the Galois-equivariant structure on $h^\ast V$ induce on the Higgs bundle $(M,\theta_M)$ in \Cref{l:pullback-of-v-vector-bundle-to-OC-nilpotent} a Galois-equivariant structure (cf \cite[Thm 3.3]{MinWang22}). In our situation,  
this can also be seen geometrically, as follows: Assume that $X$ admits a toric chart, inducing a prismatic lift $(A,I)$. Let $G_A$ be the automorphism group of the section $X\to X^\HT$ defined by $(A,I)$ like in 
\cite[Section 6.4]{AHLB}. We consider the induced toric cover $X_{\O_C,\infty}\to X_{\O_C}\to X$. The Galois group of its adic generic fibre over $\X$ is a semi-direct product $0\to \Delta\to \Gamma\to G_K\to 0$ where $\Delta=\Gal(\X_{C,\infty}|\X_{C})=\Z_p(1)^d$. We thus obtain a commutative diagram of stacks
\begin{equation}\label{eq:Galois-equivariance-of-correspondences}
\begin{tikzcd}
	{[X_{\O_C,\infty}/\Delta]} \arrow[d] \arrow[r] & X_{\O_{C}}^{\HT} \arrow[d] & {[X_{\O_C}/ T^\sharp_{X_{\O_C}}\{1\}] }\arrow[d] \arrow[l,"\sim"'] & X_{\O_C}\arrow[l]  \arrow[d]\\
	{[X_{\O_C,\infty}/\Gamma]}  \arrow[r]                              & X^\HT                      & {[X/G_A]  }\arrow[l,"\sim"']          & X\arrow[l]  
\end{tikzcd}
\end{equation}
for which the horizontal arrows define $\mathrm{S}_{\tilde X_{\mathcal{O}_C}}$ and the vertical arrows have a compatible $G_K$-action. Explicitly, this means that $G_K$ acts $C$-semilinearly on $M$ in a way that commutes with $\theta_M:M\to M\otimes \Omega^1_X(-1)$, where the Galois action on $\Omega^1_X(-1)$ is via the Tate twist. We deduce from this the following basic compatibility between the arithmetic and geometric $p$-adic Simpson functors:
\begin{proposition}\label{l:comp-geom-arithm-correspondence}
	There is a canonical natural transformation making the following diagram \mbox{$2$-commutative}:
	\[\begin{tikzcd}
		{\left\{ \begin{array}{@{}c@{}l}
				\text{$G_K$-equivariant}\\
				\text{$z$-Hitchin-small}\\\text{Higgs bundles on $\X_C$}\end{array}\right\}} \arrow[r, "{\mathrm{S}_{\tilde X_{\mathcal{O}_C}}}"] & {\left\{ \begin{array}{@{}c@{}l}
			\text{$G_K$-equivariant}\\
		\text{$z$-Hitchin-small}\\\text{v-vector bundles on $\X_C$}\end{array}\right\}}  \\
\left\{ \begin{array}{@{}c@{}l}\text{Hitchin-small}\\\text{Higgs--Sen bundles on $\X$}\end{array}\right\}\arrow[r,"\mathrm S_\pi"]\arrow[u,"h^\ast"]& \left\{ \begin{array}{@{}c@{}l}\text{v-vector bundles}\\\text{on $\X$}\end{array}\right\}     \arrow[u,"h^\ast"]    
	\end{tikzcd}\]
\end{proposition}
\begin{proof}
	The vertical arrows are well-defined by \Cref{l:pullback-of-v-vector-bundle-to-OC-nilpotent}. The natural transformation making the diagram commutative then exists by comparing \eqref{eq:Galois-equivariance-of-correspondences} to the definitions of the  $p$-adic Simpson correspondences. Namely, $S_\pi$ is given by forming the pullback of vector bundles along the bottom row, whereas  $\mathrm{S}_{\tilde X_{\mathcal{O}_C}}$ is given by forming the pullback of vector bundles along the top row.
\end{proof}
\subsection{Sen theory via the Kummer tower}
Recall that  the goal of this section is to determine the essential image of the bottom functor in \Cref{l:comp-geom-arithm-correspondence}. Since the map on the right is clearly fully faithful, and ${\mathrm{S}_{\tilde X_{\mathcal{O}_C}}}$ is an isomorphism by \Cref{t:essential-surjectivity-geom-case}, it suffices to determine the essential image of the map on the left.
Our central object of study in this section is therefore the $G_K$-module $M$. We can regard this as a v-vector bundle on $\X$ that becomes trivial upon pullback along $h$. We can describe this bundle using the stack $\Spf(\O_K)^\HT$ studied in \cite{analytic_HT}, as follows: Set $G:=G_A$.
The natural morphism $X^\HT\to \Spf(\O_K)^\HT$, induced by $X\to \Spf(\O_K)$ due to the functoriality of the Hodge--Tate stack,  induces a commutative diagram
\[\begin{tikzcd}
{[X_{\O_C,\infty}/\Gamma]}\arrow[r]\arrow[d]&	X^\HT \arrow[d]        & {[X/G]} \arrow[l,"\sim"'] \arrow[d]    & X\arrow[l] \arrow[d,equal] \\
		{[X_{\O_C}/G_K]}\arrow[r]&\Spf(\O_K)^\HT\times X & {[X/H]} \arrow[l,"\sim"']    & X\arrow[l]
\end{tikzcd}\]
where $H$ is the automorphism group of the morphism $\Spf(\O_K)\to \Spf(\O_K)^\HT$ induced by the choice of uniformiser $\pi$.
Here the third vertical map is the pushout along $G\to H$, which admits a canonical splitting defined by the section $H\to G$. Given a Higgs--Sen module $(N,\theta_N,\phi_N)$ on $[X/G]$ with $V=\mathrm S_\pi(N,\theta_N,\phi_N)$, it follows by comparing to the diagram \eqref{eq:Galois-equivariance-of-correspondences} that the Galois module $M$ on $\X_{C}$ agrees with the Galois module $h^\ast N$ associated to the Sen module $(N,\phi_N)$ on $[X/H]$.

Let now $\pi^\flat=(\dots,\pi^{1/p},\pi)$ be a choice of a compatible system of $p$-power roots of $\pi$. This induces an element $[\pi^\flat]\in A_\inf(\O_C)$. We now recall from \cite[Lemma~4.1]{AHLB}  the 1-cocycle
\[\chi_{\pi^\flat}:G_K\to \O_C^\times, \quad \chi_{\pi^\flat}(\sigma):=\theta(\tfrac{E(\sigma([\pi^\flat]))}{E([\pi^\flat])})\in 1+e\pi (1-\zeta_p)\O_C.\]
We refer to \cite[Lemma~4.1]{AHLB} for some more background on this. For us, its relevance stems from:
\begin{lemma}\label{l:Galois-action-on-vVB-assoc-to-Higgs--Sen}
	If $V=\mathrm{S}_\pi(N,\theta_N,\phi_N)$ comes from a Hitchin--small Higgs--Sen module on $X$, then there is a natural isomorphism $M=N\otimes_{\O_K}\O_C$ with respect to which the Galois action by $G_K$ is for any $\sigma \in G_K$ given by 
	\[\sigma (m\otimes c)=\chi_{\pi^\flat}(\sigma)^{\tfrac{\phi_N}{e}}(m\otimes 1)\sigma(c).\]
	Here the Hitchin-smallness implies that $\chi_{\pi^\flat}(\sigma)^{\tfrac{\phi_N}{e}}:=\exp(\tfrac{\phi_N}{e}\log(\chi_{\pi^\flat}(\sigma)))$
	converges inside $\End(M)$.
\end{lemma}
\begin{proof}
	As we have just seen, we may use the section $H\to G$ to set $\theta_N$ to $0$ without changing the Galois module structure on $M$, and we may thus assume that $V$ comes from a vector bundle on $\Spf(\O_K)^\HT\times X$ via pullback in the above diagram.
	
	In this case, the statement follows by the same proof as \cite[Lemma~4.5]{analytic_HT} by base-changing the whole discussion there along $X\to \Spf(\O_K)$.
\end{proof}
\subsection{Sen theory via the cyclotomic tower}

We now use that there is a good formalism of Sen theory for the tower $\X_{C}\to \X$, as developed by Shimizu  and Petrov. More precisely, the decompletion they work with is with respect to the cyclotomic tower $\X_{K^\cycl}\to \X$ where $K^\cycl$ is the completion of $K(\zeta_{p^n},n\in \N)$.
\begin{proposition}[{\cite[Lemma~2.10]{Shimizu_constancy}, \cite[Proposition~3.2]{petrov2020geometrically}}]\label{p:sen-theory-Shimizu-Petrov}
Let $M$ be a $G_K$-equivariant vector bundle on $\X_C$. Then there is a finite extension $K'=K(\zeta_{p^n})|K$ for which there is a vector bundle $N'$ on $\X_{K'}$ with an $\O_{\X_{K'}}$-linear endomorphism $\phi':N'\to N'$ and an isomorphism \[M=N'\otimes_{\O_{\X_{K'}}}\O_{\X_{C}}\]
 with respect to which the $G_{K'}$-action on $M$ can be described as being for any $\sigma \in G_{K'}$ given by
\[ \sigma (x\otimes c)=\chi(\sigma)^{\phi'}\cdot \sigma(c)\]
where $\chi:G_K\to \Z_p^\times$ is the cyclotomic character.
For fixed $K'$ large enough, the pair of the submodule $N'\subseteq M$ and the endomorphism $\phi'$ is uniquely determined by this property. The characteristic polynomial of $\phi'$ has coefficients in $\O(\X)$ and does not depend on the choice of $K'$.
\end{proposition}
 We will clarify the relation to \Cref{l:Galois-action-on-vVB-assoc-to-Higgs--Sen}  in \S\ref{s:comparing-Sens}.
\begin{proof}
	By \cite[Corollary~6.11]{heuer-sheafified-paCS}, we have v-descent of vector bundles along the map $\X_{C}\to \X_{K^\cycl}$, and it therefore suffices to prove the result for Galois equivariant vector bundles $M'$ on $\X_{K^\cycl}$. For these, the existence and description of the action are given by \cite[Proposition~3.2]{petrov2020geometrically}.
	
	Since strictly speaking,  \cite[Lemma~2.10]{Shimizu_constancy} (which is used in the proof) is written only in the case that $K$ is a local field, we note that one can also check directly that for $H=\Gal(K^\cycl|K)$, any continuous $1$-cocycle $H\to \GL_n(\O(X_{\O_{K^\cycl}}))$ that becomes trivial modulo $p^\alpha$ for $\alpha>\tfrac{2}{p-1}$ is in the image of
	\[ \Hom(H,\GL_n(\O(X)))\to H^1_{\cts}(H,\GL_n(\O(X_{\O_{K^\cycl}}))),\]
	by following the argument in \cite[Lemma~1]{faltings2005p}.

	The uniqueness of $(N',\phi')$ follows from considering $H$-finite vectors in $M'$, where we can recover $\phi'$ as $\log(\sigma)/\log(\chi(\sigma))$, see \cite[Thm 7.4]{MinWang22}.  
	
	A quick way to see the statement about the characteristic polynomial is to simply deduce this from classical Sen theory by specialising at all points $\Spf(\O_L)\to X$ valued in unramified extensions $L|K$. This also gives another way to see that $\phi'$ is uniquely determined.
\end{proof}

\subsection{The arithmetic Hitchin fibration for v-vector bundles}

\begin{definition}\label{def:arithmetic-Hitchin}
	Let $\mathrm{Bun}_{n}(\X_v)$ be the groupoid of v-vector bundles of rank $n$ on $\X$. We define the \textit{arithmetic Hitchin fibration} of rank $n$
	\[ \widetilde{\mathcal H}=\widetilde{\mathcal H}_\X:\mathrm{Bun}_{n}(\X_v)\to \mathbb{A}^n(\X)\]
	of $\X$ as follows: For any v-vector bundle $V$ on $\X$, let $(N,\theta_N)=\mathrm{S}^{-1}_{\tilde X_{\mathcal{O}_C}}(h^{\ast}V)$ be the associated $G_K$-equivariant Higgs bundle on $\X_{C}$ from \Cref{l:comp-geom-arithm-correspondence}. Let $(N',\phi')$ be the Sen module associated to $N$ by \Cref{p:sen-theory-Shimizu-Petrov}, then we define $\widetilde{\mathcal H}(V)$ to be the characteristic polynomial of $e\phi'$. This is well-defined as it does not depend on the choice of $K'$.
\end{definition}
\begin{definition}
We call a v-vector bundle $V$ on $\X$ \textit{Hitchin--small} if $\widetilde{\mathcal H}(V)\in\mathbb{A}^{\mathrm{sm},n}(\X)\subseteq \mathbb{A}^n(\X)$.
\end{definition}
\begin{lemma}\label{l:Hitchin-small-in-terms-of-topnil}
	$V$ is Hitchin-small if and only if
	 $e^p(\phi'^p-\phi')$ is topologically nilpotent.
\end{lemma}
\begin{proof}
	This follows from
\Cref{l:small-locus-arithmetic-Hitchin}.
\end{proof}
The technical key property of Hitchin-small v-vector bundles will be the following:
\begin{lemma}\label{l:1-cocycle-from-Hitchin-sm-v-VB}
	Let $V$ be a Hitchin-small v-vector bundle. Then for any $\sigma\in G_K$, 
	\[\textstyle \chi_{\pi^\flat}(\sigma)^{\phi'}=\sum\limits_{n=0}^\infty\frac{(\chi_{\pi^\flat}(\sigma)-1)^n}{e^n\cdot n!}\prod\limits_{i=0}^{n-1}(\phi'e-e\cdot i)\in \Aut_{C}(M)\]
	is well-defined. This defines a 1-cocycle $\chi_{\pi^\flat}^{\phi'}:G_K\to \Aut_{C}(M)$.
\end{lemma}
\begin{proof}
	The equality holds by \cite[Lemma~2.14]{analytic_HT}. Since \Cref{l:Hitchin-small-in-terms-of-topnil} shows that the products on the right hand side converge to $0$ for $n\to \infty$ when $V$ is Hitchin-small, this shows the convergence.
\end{proof}
\begin{lemma}\label{l:functoriality-arithmetic-Hitchin}
	The arithmetic Hitchin fibration is functorial, in the following sense: Let $L|K$ be a finite extension with uniformiser $\pi'\in \O_L$ and associated element $e'\in \delta_{\mathcal{O}_L/\Z_p}$. Then for any morphism $f:Y\to X$ of $p$-adic formal schemes where $X$ is smooth over $\O_K$ with rigid generic fibre $\mathcal X$ and $Y$ is smooth over $\O_L$ with rigid generic fibre $\mathcal Y$, the following diagram commutes
	\[
	\begin{tikzcd}
		{\mathrm{Bun}_{n}(\mathcal Y_v)} \arrow[r,"\widetilde{\mathcal H}_{\mathcal Y}"]                     & \mathbb A^n(\mathcal Y)      &       (\frac{e'}{e}f(a_1),\dots,(\frac{e'}{e})^nf(a_n))        \\
		{\mathrm{Bun}_{n}(\mathcal{X}_v)} \arrow[r,"\widetilde{\mathcal H}_{\mathcal X}"] \arrow[u, "f^\ast"] & \mathbb A^n(\mathcal X) \arrow[u]&(a_1,\dots,a_n)\arrow[u,mapsto].
	\end{tikzcd}\]
	where the right map in the square is defined as indicated on the right.
\end{lemma}
\begin{proof}
	Immediate from the functoriality properties of the operator $\phi'$ in \cite[Proposition~3.2]{petrov2020geometrically}, which follow from the unique characterisation in \Cref{p:sen-theory-Shimizu-Petrov}. The factor $e'/e$ appears as $\widetilde{\mathcal H}$ is defined in terms of $e\phi'$, which results in the indicated rescaling of the characteristic polynomials.
\end{proof}

We have the analogue of \Cref{l:locally-Hitchin-small}.
\begin{lemma}\label{l:locally-Hitchin-small-arithmetic-case}
 	Any v-vector bundle $V$ on $\X$ becomes Hitchin-small after pullback to the base-change $\X_L$ for some finite extension $L|K$. In fact, one can always take $L=K(\pi^{1/p^n})$ for $n\gg 0$.
\end{lemma}
\begin{proof}
	  This follows from \Cref{l:functoriality-arithmetic-Hitchin} because the locus $\mathbb A^{\sm,n}\subseteq \mathbb A^n$ is independent of the base field as soon as the base field is ramified over $W(k)$, but $(e')=\delta_{L|\Z_p}$ becomes more and more divisible by $p$ as $L$ becomes more ramified.
\end{proof}

\subsection{Comparing Sen theories for the two towers}\label{s:comparing-Sens}
We now clarify the relation between the Sen operators $\phi_N:N\to N$ and $\phi':N'\to N'$ in the case that $V=S_{\tilde X}(N,\theta_N,\phi_N)$ for some Hitchin-small Higgs--Sen bundle. A priori, these are defined using a ``Sen theory'' formalism for two different towers: The first is defined in terms of the Kummer tower $K_\infty:=K(\pi^{1/p^\infty})^\wedge|K$, the second using the cyclotomic tower $K^\cycl$. There is, however, a non-canonical way to compare the two: For this, we need to extend both operators, $\phi_N$ and $\phi'$, in a $C$-linear way to $N\otimes \O_{X_{\O_C}}$ and  $N'\otimes \O_{X_{\O_C}}$. This way, we are able to regard them as $C$-semilinear endomorphisms of the same module $M$. Then we have the following comparison (which is closely related to \cite[Thm 7.12]{MinWang22}):
\begin{lemma}\label{l:comp-between-Kummer-and-cycl}
	We have \[\phi_N=e\phi'\] inside $\End(M)$. Furthermore,
	after choosing $K'$ large enough, there exists a (non-canonical) element $z'\in 1+e(1-\zeta_p)\O_C$ for which $w:=z'^{\phi_N/e}$ in $\End(M)$ sends $N$ isomorphically onto $N'$.
\end{lemma}
\begin{proof}
	We first note that since $\phi_N/e$ is stable under extending the base field, we may without loss of generality replace $K$ by a finite extension. 
	By \cite[Lemma~3.6 and Remark~3.8]{analytic_HT}, we can then find $z'\in 1+e(1-\zeta_p)\O_C$ satisfying
	\begin{equation}\label{eq:chi-vs-chi_piflat}
	\chi_{\pi^\flat}(\sigma) \sigma(z')=\chi(\sigma)z'
	\end{equation}
	for every $\sigma \in G_K$.
	Let now $x\in N$, then by \Cref{l:Galois-action-on-vVB-assoc-to-Higgs--Sen}, we have for any $\sigma\in G_K$ that
	\[ 	\sigma(w\cdot x)=(\chi_{\pi^\flat}(\sigma)\sigma(z'))^{\phi_N/e}(x)=(\chi(\sigma)z')^{\phi_N/e}(x)=\chi(\sigma)^{\phi_N/e}(w\cdot x).\]
	It follows from the uniqueness part of \Cref{p:sen-theory-Shimizu-Petrov} that $(N',\phi')=(w\cdot N,\phi_N/e)$ inside $M$. After extending $\O_C$-linearly, this shows that $\phi'w=w\phi_N/e$. Since $w=z'^{\phi_N/e}$ commutes with $\phi_N/e$, it follows that $\phi'=\phi_N/e$ inside $\End(M)$.
\end{proof}
In particular, even though the Hitchin fibration was defined in terms of Sen theory for the Kummer tower, we still have the following, which is the reason for the rescaling by $e$ in \Cref{def:arithmetic-Hitchin}.

\begin{corollary}\label{cor:compairing-Hitchin-fibrations-arithm-case}
	There is a commutative diagram of groupoids
	\[
	\begin{tikzcd}[column sep =0.4cm,row sep=0.2cm]
		{\left\{ \begin{array}{@{}c@{}l}\text{Hitchin-small rank $n$}\\\text{Higgs--Sen modules on $\X$}\end{array}\right\}} \arrow[rd,"\mathcal H"'] \arrow[rr,"\mathrm{S}_\pi"] &         & {\left\{ \begin{array}{@{}c@{}l}\text{v-vector bundles}\\\text{on $\X$ of rank $n$}\end{array}\right\}} \arrow[ld,"\widetilde{\mathcal H}"] \\
		& \mathbb A^n(\X) &              
	\end{tikzcd}\]
\end{corollary}

\subsection{The essential image in the arithmetic case}
We are now ready to prove the main result:
\begin{theorem}\label{t:essential-image-vVB-arithmetic}
	Let $X$ be a $p$-adic smooth formal scheme over $\O_K$. The essential image of $\mathrm{S}_\pi$ is given by the Hitchin-small v-vector bundles, i.e.\ we have an equivalence of categories
	\[ \mathrm{S}_\pi:{\left\{ \begin{array}{@{}c@{}l}\text{Hitchin-small}\\\text{Higgs--Sen modules on $\X$}\end{array}\right\}}\isomarrow {\left\{ \begin{array}{@{}c@{}l}\text{Hitchin-small}\\\text{v-vector bundles on $\X$}\end{array}\right\}}.\]
\end{theorem}
\begin{proof}
	That the essential image is contained in the Hitchin-small v-vector bundles follows from \Cref{cor:compairing-Hitchin-fibrations-arithm-case}.
	To see the converse, it suffices to work locally on $X$, and we may therefore assume that $X$ admits a toric chart. Let $V$ be a Hitchin-small v-vector bundle on $\X$ of rank $n$. Let \[(M,\theta_M):=S_{\tilde X_{\O_C}}^{-1}(V)\] be the associated $G_K$-equivariant Higgs bundle on $\X_{C}$ from \Cref{l:pullback-of-v-vector-bundle-to-OC-nilpotent}. Let $(N',\phi')$ be the Sen module associated to the $G_K$-module $M$ by \Cref{p:sen-theory-Shimizu-Petrov}. Then the assumption that $V$ is Hitchin-small means by definition that $e^p(\phi'^{p}-\phi')$ is topologically nilpotent.
	
	With the preparations from this section, we can now implement the strategy of \cite[Lemma~4.6]{analytic_HT}: By \Cref{l:1-cocycle-from-Hitchin-sm-v-VB}, the assumption on $V$ ensures that  the series
	\[ \chi_{\pi^\flat}(\sigma)^{\phi'} \in \End_C(M)\]
	converges and defines a continuous $1$-cocycle $\chi_{\pi^\flat}^{\phi'}:G_K\to \Aut_C(M)$. We can use this to define a new semi-linear $G_K$-action on $M$ given on $m\in M$ by
	\[ \sigma \ast m:=\chi_{\pi^\flat}(\sigma)^{-\phi'}\cdot \sigma(m).\] 
	Let now $z'$ and $K'$ be as in \Cref{l:comp-between-Kummer-and-cycl} and set $w:=z'^{\phi'}\in \Aut_C(M)$. Then for any $x\in N'\subseteq M$ and $\sigma\in G_{K'}$, 
	\[ \sigma \ast (w^{-1}x)=\chi_{\pi^\flat}(\sigma)^{-\phi'}  (\sigma(z')^{-1}\chi(\sigma))^{\phi'}(x)=w^{-1}(x).\]
	Thus $w^{-1}(N')\subseteq M$ is a finite free $\O(\X_{K'})$-submodule of rank $n$ that is fixed by the $\ast$-action of $G_{K'}$. It follows by Galois descent along the finite \'etale cover $\X_{K'}\to \X$ that
	\[N:=M^{G_K,\ast},\]
	the $G_K$-invariants under the $\ast$-action, are still a finite projective $\O(\X)$-module of rank $n$. Hence the Galois descent of $M$ for the $\ast$-action along $\X_{C}\to \X_{K}$ is effective, namely $M$ is isomorphic as a Galois-module to the pullback of $N$. 

	Note that the original $G_K$-action on $N$ is given for $\sigma\in G_K$ by
	\[ \sigma(x)=\chi_{\pi^\flat}(\sigma)^{\phi'}(x).\]
	Since $\phi'$ preserves $N'$ and $w^{-1}=z'^{-\phi'}$ clearly commutes with $\phi'$, we see that $\phi'$ preserves $N$.  Finally,
	since the Tate twist $\O(1)$ corresponds to the Sen module $(\O,1)$ by \Cref{p:sen-theory-Shimizu-Petrov}, we see by taking Galois invariants for the $\ast$-action that the Higgs field $\theta_M:M\to M\otimes \Omega^1_{X_{\O_C}}(-1)$ descends to a morphism $\theta_N:N\to N\otimes \Omega^1_X\{-1\}$ of Sen modules, where the Sen operator on  $N\otimes \Omega^1_X\{-1\}$ is given by $\phi'\otimes 1+\id$.

	We can therefore set $\phi_N:=e\phi'\in \End(N)$ to obtain a Higgs--Sen module  $(N,\theta_N,\phi_N)$ on $\X$. We claim that now $\mathrm{S}_\pi(N,\theta_N,\phi_N)=V$. This follows from \Cref{l:comp-geom-arithm-correspondence} because by construction, we have $h^{\ast}(N,\theta_N,\phi_N)=(M,\theta_M)$, hence
	\[h^{\ast}\mathrm{S}_{\pi}(N,\theta_N,\phi_N)=\mathrm{S}_{\tilde X_{\mathcal{O}_C}}(h^{\ast}(N,\theta_N,\phi_N))=\mathrm{S}_{\tilde X_{\mathcal{O}_C}}(M,\theta_M)=h^\ast V,\]
	 and the vertical arrows in the diagram in \Cref{l:comp-geom-arithm-correspondence} are clearly fully faithful.
\end{proof}
This finishes the proof of the main result of this subsection. We now give some applications. Firstly, we deduce that we can describe all v-vector bundles when we allow finite extensions of $K$:
\begin{corollary}\label{cor:arithm-vVB-locally-Hitchin-small}
	Let $X$ be a quasi-compact smooth $p$-adic formal scheme over $\O_K$. Let $V$ be a v-vector bundle on $\X$. Then there is a finite field extension $L|K$ such that for $f:\X_L\to \X_K$, we have  $f^\ast V=\mathrm S_{\pi}(N,\theta_N,\phi_N)$ for some Higgs--Sen module $(N,\theta_N,\phi_N)$ on $\X_{L}$.
	In fact, we can always take $L=K(\pi^{1/p^n})$ for $n$ large enough.
\end{corollary}
\begin{proof}
	This follows from combining \Cref{t:essential-image-vVB-arithmetic} and \Cref{l:locally-Hitchin-small-arithmetic-case}.
\end{proof}
\subsection{Applications}
	One can use this to prove analogues in the arithmetic setting of some of the analyticity criteria for v-vector bundles of \cite[\S7]{heuer-sheafified-paCS}. For example:
\begin{corollary}
	Let $X$ be a smooth $p$-adic formal scheme over $\O_K$ and let $V$ be a $v$-vector bundle on the  rigid generic fibre $\X$. Then the following are equivalent:
	\begin{enumerate}
		\item $V$ is analytic-locally trivial.
		\item For each point  $x\colon\Spf(\O_L)\to X$ valued in a finite extension of $K$, the Sen operator associated to $x^{\ast}V$ is trivial.
		\item For each point $x\colon\Spf(\O_L)\to X$ valued in a finite extension of $K$, the pullback $x^{\ast}V$ is trivial.
	\end{enumerate}
\end{corollary}
\begin{proof}
	The implications  $(1)\Rightarrow (3)\Rightarrow (2)$ are clear. To see $(2) \Rightarrow (1)$, we can localise on $X$  to assume that $X$ is affine. Then by \Cref{cor:arithm-vVB-locally-Hitchin-small} we can  further pass to a finite extension of $K$ so that $V=\mathrm S_{\pi}(N,\theta_N,\phi_N)$. Then condition (2) means that $\phi_N$ vanishes on points. Since $X$ is reduced, this shows $\phi_N=0$. Due to the relation $[\phi_N,\theta_N]=-E'(\pi)\theta_N$ in the sheaf $\mathcal{HS}_{\X,\pi}$ before \cite[Proposition 6.36]{AHLB}, this implies that $\theta_N=0$. Hence $V$ is analytic-locally trivial.
\end{proof}

As another application, we obtain a generalisation of \cite[Thm~1.2]{analytic_HT}, which is the case of $X=\Spf(\O_K)$, to the case of smooth $p$-adic formal schemes over $\O_K$.
\begin{corollary}
	Let $X$ be a quasi-compact smooth $p$-adic formal scheme over $\O_K$. Then there is an equivalence of categories
	\[2\text{-}\varinjlim\limits_{L|K} \Bun([X_{\O_L}^{\HT}/\Gal(L|K)],\O\tf)\to  \Bun(\X_v) \]
	where $L|K$ runs through the finite Galois extensions inside a fixed algebraic closure of $K$.
\end{corollary}
\begin{proof}
	As the fully faithful functor $\Bun(X_{\O_L}^{\HT},\O\tf)\to \Bun(\X_{L,v})$ from \eqref{eq:alpha-ast-Bun(XHT)} in \S\ref{s:recollections} is completely canonical and natural in $\X_{L}$, it defines by Galois descent a fully faithful functor \[\Bun([X_{\O_L}^{\HT}/\Gal(L|K)],\O\tf)\to\Bun(\X_v).\] These functors are compatible as $L$ varies and thus define the desired functor in the 2-colimit. To see that this is essentially surjective, it suffices by \Cref{t:essential-image-vVB-arithmetic} to see that any v-vector bundle on $\X_v$ becomes Hitchin-small after pullback to $\X_{L}$ for some $L$. This follows from \Cref{l:locally-Hitchin-small-arithmetic-case}.
\end{proof}
As a third application, we can answer a question of Min--Wang \cite[Remark~7.25]{MinWang22}:
\begin{corollary}
	For  a v-vector bundle $V$ on $\X$, the following are equivalent:
	\begin{enumerate}
		\item $V$ lies in the essential image of $\mathrm{S}_\pi$.
		\item For every point $x:\Spa(L)\to \X$  with values in a finite extension $L|K$, the Galois representation $V_x$ is ``nearly Hodge--Tate'' in the sense of \cite{gao2022hodge}, i.e.\ the Hodge--Tate--Sen weights of $V_x$ all lie in   $\Z+\delta_{\mathcal{O}_L/\Z_p}^{-1}\cdot \mathfrak{m}_{\overline{L}}$. 
	\end{enumerate}
\end{corollary}
\begin{proof}
	Since the Hodge--Tate--Sen weights are by definition the eigenvalues of the cyclotomic Sen operator, it is clear from \Cref{def:arithmetic-Hitchin} and \Cref{l:small-locus-arithmetic-Hitchin} that the v-vector bundle $V_x$ on $\Spa(L)$ is Hitchin-small if and only if it is nearly Hodge--Tate.
	The implication (1) $\Rightarrow$ (2) thus follows from functoriality of the arithmetic Hitchin fibration, \Cref{l:functoriality-arithmetic-Hitchin}, applied to the morphism $x$.  The direction (2) $\Rightarrow$ (1) follows from \Cref{t:essential-image-vVB-arithmetic}, as being Hitchin small means that $\widetilde{\mathcal H}(V)\in \O(\X)^n$ lies in some open subspace, and this can be checked on points due to functoriality of $\widetilde{\mathcal H}$.
\end{proof}


\appendix 
\section{$p$-adic Simpson correspondences -- a one-page overview}
	\begin{center}
		{
			\setlength{\tabcolsep}{3pt}
			\begin{tabular}{ | c || c | c | c | c | }
				\hline
				&\makecell{\textbf{local}\\ \textbf{$p$-adic Simpson}} &\makecell{\textbf{small}\\ \textbf{$p$-adic Simpson}} & \makecell{\textbf{proper} \\\textbf{$p$-adic Simpson}} & \makecell{\textbf{arithmetic}\\ \textbf{$p$-adic Simpson}} \\\hline\hline 
				\makecell{non-archimedean\\
				base field $K|\Q_p$}& perfectoid& perfectoid&\makecell{algebraically\\ closed}&\makecell{discretely valued\\ with perfect \\residue field}\\\hline 
				\makecell{smooth rigid\\ space $\X$ over $K$} &\makecell{$\X$ is toric\\ affinoid}&\makecell{$\X$ has a smooth\footnote{\label{footnote1}More generally, this should work for $X$ of semi-stable reduction, see the references to \cite{faltings2005p}, \cite{abbes2016p}, \cite{MinWang22semistable}.} \\formal model $X$\\ with lift to $A_2(x)$}& $\X$ is proper&\makecell{$\X$ has a smooth\cref{footnote1}\\ formal model $X$}\\\hline
				\makecell{Betti side of\\ correspondence} &\makecell{Faltings-small\\ v-vector bundles} &\makecell{Hitchin-small\\ v-vector bundles}&\makecell{v-vector bundles}&\makecell{Hitchin-small\\ v-vector bundles}\\\hline
				\makecell{Dolbeault side of\\ correspondence} &\makecell{Faltings-small\\ Higgs bundles} &\makecell{Hitchin-small\\ Higgs bundles}&\makecell{Higgs bundles}&\makecell{Hitchin-small\\ Higgs--Sen bundles}\\\hline
				\makecell{choice of\\lifting data} &toric chart of $\X$&
				\makecell{ $A_2(x)$-lift of  \\ formal model $X$ } & $B_2$-lift of $\X$&\makecell{ uniformiser $\pi$ \\ of $\O_K$ with 
					\\$\pi^\flat$ of $\pi$ along\\ $\sharp:\O_C^\flat\to \O_C$}\\\hline
				\makecell{essential \\further choices} &none& none & 	\makecell{exponential \\map for $K$}&none\\\hline
				\makecell{naturality \\for given\\ $f\colon\mathcal Y\to \mathcal X$} &\makecell{morphism of\\ charts induces\\ canonical natural\\ equivalence}& \makecell{canonically \\natural with\\ respect to\\twisted pullback}   & \makecell{non-canonically\\ natural}&\makecell{canonically\\ natural given \\compatible $\pi^\flat$,\\ otherwise only\\ non-canonically \\natural}\\\hline
				
				\makecell{reference for \\correspondence\\as  stated above,\\ then further \\related references} &	\makecell{\cite[Thm~6.5]{heuer-sheafified-paCS},\\\cite[Thm~3]{faltings2005p}\\\cite[\S II.13]{abbes2016p}\\\cite[Thm~13.7]{Tsuji-localSimpson}}&\makecell{Thm \ref{t:essential-surjectivity-geom-case},\\\cite[Thm~5]{faltings2005p}\\\cite[Thm III.12.26]{abbes2016p}\\\cite[Thm 5.3]{wang2021p}\\\cite[Thm~6.4]{MinWang22}\\\cite[Thm~4.10]{MinWang22semistable}} &\makecell{ \cite[Thm~1.1]{heuer-proper-correspondence},\\\cite[Thm~6]{faltings2005p}\\\cite[Thm~1.3]{heuer-v_lb_rigid}}&\makecell{Thm \ref{t:essential-image-vVB-arithmetic},\\\cite[Thm 15.1]{Tsuji-localSimpson}\\\cite[Thm~2.1]{LiuZhu_RiemannHilbert}\\\cite[Thm~3.3, \S7]{MinWang22}}\\\hline
		\end{tabular}   }
	\end{center}

\end{document}